\documentclass{article}

\usepackage[margin=1in]{geometry}
\usepackage[toc,page]{appendix}
\usepackage{etoolbox}
\appto\appendix{\addtocontents{toc}{\protect\setcounter{tocdepth}{0}}}
\usepackage[OT2,OT1]{fontenc}
\usepackage{lmodern}
\usepackage{amsmath,amssymb,amsfonts,amsthm,amscd,mathtools}
\usepackage{mathrsfs}
\usepackage{stmaryrd}
\usepackage{dsfont}
\usepackage{multirow}
\usepackage{enumitem}
\usepackage{comment}
\usepackage{verbatim}
\usepackage{xcolor}
\usepackage{graphicx}
\usepackage{float}
\usepackage{upgreek}
\usepackage{colonequals}
\usepackage{romannum}
\usepackage{relsize}
\usepackage{url}
\usepackage{tikz}
\usepackage{tikz-cd}
\usetikzlibrary{cd,arrows,matrix,positioning}
\usepackage[all,cmtip]{xy}
\usepackage[colorlinks=true,linkcolor=black,citecolor=black,urlcolor=black]{hyperref}

\numberwithin{equation}{section}
\numberwithin{table}{section}
\theoremstyle{plain}
\newtheorem{theorem}{Theorem}[section]

\newtheorem{lemma}[theorem]{Lemma}
\newtheorem{proposition}[theorem]{Proposition}

\newtheorem{corollary}[theorem]{Corollary}

\newtheorem*{theorem*}{Theorem}
\newtheorem*{definition*}{Definition}
\newtheorem*{claim*}{Claim}
\theoremstyle{definition}
\newtheorem{definition}[theorem]{Definition}

\newtheorem{example}[theorem]{Example}
\theoremstyle{remark}
\newtheorem{remark}[theorem]{Remark}

\newif\ifshowcomments
\showcommentstrue

\DeclareFontEncoding{OT2}{}{}

\makeatletter
\newcommand{\setcmd}[2]{%
  \@ifundefined{#1}
    {\expandafter\newcommand\csname #1\endcsname{#2}}
    {\expandafter\renewcommand\csname #1\endcsname{#2}}}
\makeatother

\newcommand{\bA}{\mathbb{A}}
\newcommand{\bB}{\mathbb{B}}
\newcommand{\bC}{\mathbb{C}}
\newcommand{\bD}{\mathbb{D}}
\newcommand{\bE}{\mathbb{E}}
\newcommand{\bF}{\mathbb{F}}
\newcommand{\bG}{\mathbb{G}}

\newcommand{\bK}{\mathbb{K}}

\newcommand{\bN}{\mathbb{N}}

\newcommand{\bP}{\mathbb{P}}
\newcommand{\bQ}{\mathbb{Q}}
\newcommand{\bR}{\mathbb{R}}
\newcommand{\bS}{\mathbb{S}}
\newcommand{\bT}{\mathbb{T}}

\newcommand{\bX}{\mathbb{X}}
\newcommand{\bY}{\mathbb{Y}}
\newcommand{\bZ}{\mathbb{Z}}

\setcmd{A}{\bA}
\setcmd{B}{\bB}
\setcmd{C}{\bC}
\setcmd{D}{\bD}
\setcmd{E}{\bE}
\setcmd{F}{\bF}
\setcmd{G}{\bG}
\setcmd{K}{\bK}
\setcmd{N}{\bN}
\setcmd{P}{\bP}
\setcmd{Q}{\bQ}
\setcmd{R}{\bR}
\setcmd{T}{\bT}
\setcmd{X}{\bX}
\setcmd{Y}{\bY}
\setcmd{Z}{\bZ}

\newcommand{\kbar}{\overline{k}}
\newcommand{\Kbar}{\overline{K}}
\newcommand{\Gm}{\mathbb{G}_m}

\newcommand{\cA}{\mathcal{A}}
\newcommand{\cB}{\mathcal{B}}

\newcommand{\cI}{\mathcal{I}}
\newcommand{\cJ}{\mathcal{J}}
\newcommand{\cK}{\mathcal{K}}
\newcommand{\cL}{\mathcal{L}}
\newcommand{\cM}{\mathcal{M}}

\newcommand{\cP}{\mathcal{P}}

\newcommand{\cU}{\mathcal{U}}

\newcommand{\cZ}{\mathcal{Z}}

\newcommand{\sB}{\mathscr{B}}

\newcommand{\sS}{\mathscr{S}}

\setcmd{s}{\sS}

\newcommand{\fm}{\mathfrak{m}}
\newcommand{\fn}{\mathfrak{n}}

\newcommand{\fp}{\mathfrak{p}}

\DeclareMathOperator{\Gal}{Gal}

\DeclareMathOperator{\Hom}{Hom}

\DeclareMathOperator{\Spec}{Spec}
\DeclareMathOperator{\Sp}{Sp}

\DeclareMathOperator{\ord}{ord}

\DeclareMathOperator{\red}{red}

\newcommand{\ol}[1]{\overline{#1}}

\newcommand{\eps}{\varepsilon}

\newcommand{\hhat}{\hat{h}}

\newcommand{\an}{\mathrm{an}}

\newcommand{\Stab}{\operatorname{Stab}}
\newcommand{\Sk}{\operatorname{Sk}}
\newcommand{\ch}{\operatorname{char}}
\newcommand{\id}{\operatorname{id}}
\newcommand{\charac}{\operatorname{char}}
\newcommand{\perf}{\mathrm{perf}}
\newcommand{\sep}{\mathrm{sep}}
\newcommand{\sm}{\mathrm{sm}}
\newcommand{\abs}[1]{\lvert#1\rvert}
\author{Fateme Sajadi}

\title{Frey-Mazur theorem for semistable elliptic curves over function fields}
\date{}
\begin{document}

\maketitle

\begin{abstract}
Let $k$ be a field, let $B/k$ be a smooth projective geometrically integral curve, and put $K=k(B)$. We prove a Frey--Mazur type theorem for nonisotrivial semistable elliptic curves over $K$. More precisely, after fixing the genus of $B$ and the characteristic of $k$, a sufficiently large prime-to-characteristic torsion representation determines the $K$-isogeny class. 
\end{abstract}

\pagenumbering{arabic}

\section{Introduction}
The Frey--Mazur philosophy originating in \cite{MazurGoldfeld1978} predicts that for sufficiently large $N$, the $N$-torsion subgroup of an elliptic curve along with its Galois structure determines the isogeny class of the elliptic curve. Over a number field, the familiar conjecture is usually stated for prime torsion.   We ask the same question for geometric function fields.  For a complex base curve, Bakker and Tsimerman proved that isomorphic prime-torsion monodromy representations of two nonisotrivial elliptic families force the families to be isogenous once the prime exceeds a constant depending only on the gonality of the base \cite[Thm~2]{bakker2016p}.  Their proof uses complex hyperbolic geometry.  In this article, we impose semistability and use a global height gap. This gives a statement over arbitrary constant fields in every characteristic, when we look at composite levels prime to the characteristic of the base. 

Let $k$ be a field. The \emph{characteristic exponent} of $k$ is
\[
 p=\begin{cases}
 1,&\ch(k)=0,\\
 \ch(k),&\ch(k)>0.
 \end{cases}
\]
Let $B$ be a geometrically irreducible, smooth projective curve over $k$ and let $K:=k(B)$.  For a closed point $v\in B$, write $d_v=[\kappa(v):k]$ and normalize the associated absolute value by
\[
 -\log |a|_v=d_v\ord_v(a),\qquad a\in K^\times.
\]
Let $|B|$ be the set of closed points of $B$. For a finite set $\Sigma\subset |B|$, put
\[
 \deg_k\Sigma=\sum_{v\in\Sigma}d_v.
\]
An elliptic curve $E/K$ is nonisotrivial if $j(E)\notin\kbar$. Equivalently, its $j$-map $j_E:B\dashrightarrow\mathbb P^1_k$ is nonconstant.  It is semistable if it has either good or multiplicative reduction at every closed point of $B$. Denote $G_K:=\Gal(K^{\sep}/K)$ and for an elliptic curve $E$ over $K$, let $E[N]$ denote the geometric $N$-torsion with its natural $G_K$-action. Note that when $(N,p)=1$ then $N$-torsion points are defined over $K^{\sep}$.

\begin{theorem}\label{thm:main}
Fix integers $b\geq 0$ and $p\geq 1$, where $p=1$ or $p$ is prime.  There is a constant $C(b,p)$ with the following property. Let $k$ have characteristic exponent $p$, let $B/k$ be a smooth projective geometrically integral curve of genus at most $b$, and put $K=k(B)$.  Let $E,F/K$ be nonisotrivial semistable elliptic curves.  If $(N,p)=1$, $N>C(b,p)$, and
\[
 \phi:E[N]\xrightarrow{\sim}F[N]
\]
is an isomorphism of $G_K$-modules, then $E$ and $F$ are isogenous over $K$.
\end{theorem}

 If $p>1$, the condition $(N,p)=1$ is indispensable since $E[p^a](\Kbar)$ need not be a free rank-two $\mathbb Z/p^a\mathbb Z$-module, and may have fewer than $p^{2a}$ geometric points. Nonisotriviality is also necessary. For example, over an algebraically closed constant field, the prime-to-$p$ torsion representations of constant elliptic curves are trivial.

\subsection{Methods of proof}
This paper was inspired by the recent work of Looper and Yap. They proved uniform boundedness of torsion on abelian varieties over characteristic-zero function fields with no isotrivial part  \cite[Thm~1.1]{looper2026uniform}.  A new local input is their uniform lower bound for Looper's determinant-valued Arakelov--Green function when all points reduce to the identity component of the N\'eron model \cite[Thm~1.5]{looper2026uniform}.  Analytically, those points retract to the origin of the canonical Berkovich skeleton.

The only property of the origin used in \cite[Thm 1.5]{looper2026uniform} is that the origin is
a proper closed set.  In this paper, we expand this idea to any proper closed subset of the skeleton under suitable conditions.  Fekete configurations for an ample metrized line bundle equidistribute to the canonical Chambert--Loir measure. On a semistable abelian variety that measure has full support on the canonical skeleton.  Therefore, configurations constrained over any proper closed part of the skeleton cannot be asymptotically Fekete.  We make this observation uniform and apply it to the diagonal inside the product  $\bS^1\times \bS^1$ coming from the graph of the isomorphism $E[N]\to F[N]$ inside the skeleton of the product $E\times F$. Here we explain the setup and the main steps of the proof. 

Berkovich analytification requires a chosen non-Archimedean absolute value and a complete ground field \cite{berkovich1990spectral}.   For a closed point $v\in |B|$, let $\operatorname{ord}_v$ be normalized to have value group $\mathbb Z$, and set
\[
  d_v=[\kappa(v):k],\qquad  |a|_v=\exp\bigl(-d_v\operatorname{ord}_v(a)\bigr)  \quad(a\in K^\times),  \qquad C_v=\widehat{\overline{K_v}}.
\]
The absolute value on $K$ extends uniquely to $C_v$.  With this normalization,
\[
  \sum_{v\in |B|}\log|a|_v=0\qquad(a\in K^\times).
\]
If $E/K$ is finite, $B_E$ is the normalization of $B$ in $E$, and $w\in |B_E|$ lies above $v$, we normalize the absolute value at $w$ by
\[
  -\log|b|_w=[\kappa(w):k]\operatorname{ord}_w(b) \qquad(b\in E^\times).
\]
Consequently, if \(e(w/v)=e_w,\) and \(f(w/v)=f_w=[\kappa(w):\kappa(v)]\), then for any $a\in K^\times$
\[
  |a|_w=|a|_v^{e_wf_w},\quad  \qquad \sum_{w\mid v}e_wf_w=[E:K].
\]

Let $(A,\lambda)/K$ be a principally polarized abelian variety. If $P_A$ denotes the normalized Poincar\'e bundle on $A\times A^\vee$, set
\begin{equation}\label{eq:intro-line-bundle}
 \mathcal M_\lambda=(\id_A,\lambda)^*P_A,  \qquad L_A=\mathcal M_\lambda^{\otimes 3}.
\end{equation}
The bundle $\mathcal M_\lambda$ represents $2\lambda$, so $L_A$ is symmetric, very ample, and normally generated. We write $h_{\omega,K}(A)$ for the stable Hodge height and $\widehat h_{A,L_A,K}$ for the associated N\'eron--Tate height. Here $\omega$ denotes the Hodge bundle and is unrelated to the ultrafilter used below.
We now describe the skeletal statement.

Fix $g$ and $p$.  Choose $\ell_p\geq 3$ prime to $p$ and put
\[
\kappa_p=\begin{cases} \mathbb Q(\zeta_{\ell_p}),&p=1,\\ \mathbb F_p(\zeta_{\ell_p}),&p>1,
\end{cases}
\]
where $\zeta_{\ell_p}$ is a primitive $\ell_p$-th root of unity. Let $X=\cA_{g,\ell_p}$ be the fine moduli scheme of principally polarized abelian varieties of dimension \(g\) with full level-\(\ell_p\) structure and let $\overline X$ be a projective toroidal compactification, and choose a finite affine cover $X=\bigcup_aU_a$ carrying the dynamical data constructed in Section \ref{sec:uniform-package}. See also \cite[Prop 6.4]{looper2026uniform}. The fixed models give boundary functions $\lambda_{\partial X}$ and $\lambda_{\partial U_a}$. A bad-reduction fiber represented in $U_a$ is centered in that chart if
\[
\lambda_{\partial U_a}(A)\le\lambda_{\partial X}(A).
\]

To define compatibility, consider centered fibers $A_j/C_j$ over complete algebraically closed non-Archimedean fields where $\kappa_p$ is trivially valued, proper closed subsets $Z_j\subsetneq S(A_j)$, and positive rescaling factors $\varepsilon_j=\lambda_{\partial X}(A_j)^{-1}$.  For a nonprincipal ultrafilter $\omega$, set
\[
  \mathcal U^{\varepsilon}=\left\{(a_j)\in\prod_j C_j: \sup_j |a_j|^{\varepsilon_j}<\infty\right\}, \qquad |(a_j)|_\omega=\lim_\omega |a_j|^{\varepsilon_j}.
\]
The quotient of $\mathcal U^{\varepsilon}$ by the kernel of this seminorm is the complete algebraically closed valued field $H(\omega)$. See Lemma~\ref{lem:ultralimit-field} below, which proves this for the varying fields used here. See Favre--Gong's construction \cite[Prop~3.9]{favre2025non}.  The relative moduli point and abelian scheme specialize to an abelian variety $A_\omega/H(\omega)$, called the \emph{fundamental ultralimit}.

\begin{definition}\label{def:compatibility}
The centered sequence $(A_j,Z_j)$ is \emph{$\omega$-compatible} if there is a proper closed set $Z_\omega\subsetneq S(A_\omega)$ such that, for every sequence $x_j\in A_j(C_j)$,
\[
  r_{A_j}(x_j)\in Z_j\text{ for every }j \quad\Longrightarrow\quad r_{A_\omega}(x_\omega)\in Z_\omega.
\]
A class of pairs is \emph{admissible} if every centered sequence in the class is $\omega$-compatible for at least one nonprincipal ultrafilter $\omega$.  An \emph{admissible system} is one such class for every chart $U_a$.
\end{definition}

\begin{definition}\label{def:universal-restriction}
Let $A/K$ be an abelian variety, let $\Sigma\subset |B|$ be finite, and let $Z_v\subset S(A_{C_v})$ be closed for every $v\in\Sigma$.  A finite set $P\subset A(\overline K)$ is \emph{universally $\Sigma$-restricted} if, for every $x\in P$, every $v\in\Sigma$, and every $K$-embedding of a finite field of definition of $x$ into $C_v$, one has
\[
  r_{A,C_v}(x)\in Z_v.
\]
Equivalently, after placing $P$ in a finite extension $E/K$, this condition holds at every place $w\mid v$.
\end{definition}

In order to pass from the fine level moduli space to an abelian variety without a chosen level structure, fix a finite separable extension $K'/K$ over which $A[\ell_p]$ is trivial.    For $w\mid v$, put $c_{w/v}=e(w/v)f(w/v)$.  Our normalization gives $|a|_w=|a|_v^{c_{w/v}}$ for $a\in K$, and hence, after choosing the corresponding extension of the valuation, we get a natural scaling homeomorphism
\[
\sigma_{w/v}:S(A_{C_v})\longrightarrow S(A_{C_w}).
\]
It multiplies tropical coordinates by $c_{w/v}$ and commutes with the canonical retractions.  We therefore transport a skeletal restriction $Z_v\subset S(A_{C_v})$ by setting
\[
  Z_w=\sigma_{w/v}(Z_v).
\]
We also ask that every transported pair $(A_{C_w},Z_w)$, endowed with the chosen level structure, belong to one of the admissible chart classes used in the uniform local theorem.  It is enough to impose this condition for one bounded level-trivializing extension $K'/K$ and every $w\mid v$.  We call a system satisfying this condition \emph{level-stably admissible}.

\begin{theorem}\label{thm:heightgap}
Fix $g\geq 1$, $M_0>0$, a characteristic exponent $p$, and a level-stably admissible system of skeletal restrictions as above.  There are constants $q\geq 2$ and $\delta>0$, depending only on these data, with the following property.  Let $(A,\lambda)/K$ be any $g$-dimensional principally polarized abelian variety, and suppose that $h_{\omega,K}(A)>0$.  Let $\Sigma$ be a finite set of semistable bad-reduction places, prescribe a proper closed set $Z_v\subsetneq S(A_{C_v})$ from the given system for every $v\in\Sigma$, and assume that
\[
\deg_k\Sigma\geq \frac{h_{\omega,K}(A)}{M_0}.
\] 
If $P\subset A(\overline K)$ is a universally $\Sigma$-restricted set of distinct points and $|P|\geq h^0(A,L^{\otimes q})$, then either a nonzero section of $L_{\overline K}^{\otimes q}$ vanishes on $P$, or
\[
\sum_{x\in P}\widehat h_{A,L,K}(x) \geq \delta h_{\omega,K}(A)>0.
\]
\end{theorem}

Note that only the selected places $\Sigma$ are required to be semistable. 

\begin{theorem}\label{thm:descent} 
Let $X/K$ be projective, let $L$ be a line bundle on $X$, and let $P\subset X(K^{\mathrm{sep}})$ be a finite $\operatorname{Gal}(K^{\mathrm{sep}}/K)$-stable set.  If a nonzero section of $L_{\overline K}$ vanishes on $P$, then a nonzero section of $L$ over $K$ vanishes on $P$.
\end{theorem}

Later in Section \ref{sec:compatibility}, we give some criteria on when the prescribed closed sets $Z_j\subset S(A_j)$ are admissible. One can modify the isomorphism $\phi:E[N]\to F[N] $ to get that for a big subset of the multiplicative reductions $v$, the graph of $\phi$ reduces to the diagonal inside $S(E_v)\times S(F_v)$. Thus we can apply the height gap Theorem \ref{thm:heightgap}. Since every element of the graph of $\phi$ consists of torsion points, the sum of their canonical heights is zero. Hence a section should vanish on all of them. We argue in Section \ref{sec divisors} that when $E, F$ are nonisogenous this would result in a bound on $N$.

\subsection*{Outline}
Section \ref{sec:product-formula} fixes the height normalizations and proves a positive determinant gap for configurations whose retractions lie in a proper closed subset of a fixed skeleton. Section \ref{sec:uniform-package} constructs dynamical charts in arbitrary characteristic and uses ultralimits to make this gap uniform for admissible collections. Section \ref{sec:compatibility} provides practical compatibility criteria. Section \ref{sec:uniform-global-proof} combines the local gaps with a global determinant estimate and proves Theorems \ref{thm:heightgap} and \ref{thm:descent}.

Section \ref{sec reductions} reduces the main theorem to perfect constant fields and separable $j$-maps. The semistable discriminant inequality of Szpiro then provides many places where both curves have multiplicative reduction with bounded discriminant exponents. Section \ref{sec admissibility} analyzes a prime-power torsion isomorphism at these places. After lowering the level by a bounded factor and multiplying the isomorphism by one global scalar, its graph retracts into a signed diagonal in each product skeleton. Then we prove that these signed diagonals form a level-stably admissible system. Section \ref{sec divisors} applies the height gap to the graph of a given isomorphism, descends the resulting divisor, and bounds the level by intersection theory unless the curves are K-isogenous. A final argument passes from prime-power levels to arbitrary levels.

\section{Heights and equidistribution}\label{sec:product-formula}\label{sec:base-change}\label{sec:local-energy}
Let $|B|$ denote the set of closed points of $B$.  Since $B$ is smooth, the local ring at every $v\in|B|$ is a discrete valuation ring.  Put
\[
d_v=\deg_k(v)=[\kappa(v):k],\qquad -\log|a|_v=d_v\ord_v(a).
\]

\begin{theorem}\cite[Thm 1.4.11]{stichtenoth2009algebraic}.\label{prop:product-formula}
For every $a\in K^\times$ one has
\[
 \sum_{v\in|B|}\log|a|_v=0.
\]
Only finitely many terms are nonzero.
\end{theorem}

Let $E/K$ be finite, let $B_E$ be the normalization of $B$ in $E$, and write $w\mid v$ for a closed point $w\in B_E$ above $v$.  Set
\[
 e_w=e(w/v),\qquad f_w=f(w/v)=[\kappa(w):\kappa(v)].
\]
We normalize the absolute values on $E$ relative to the original field $k$:
\[
 -\log|b|_w=[\kappa(w):k]\ord_w(b).
\]

\begin{proposition}\label{prop:ef-scaling}
For $a\in K^\times$ and $w\mid v$,
\begin{equation}\label{eq:ef-scaling}
 -\log|a|_w=e_wf_w(-\log|a|_v),
 \qquad |a|_w=|a|_v^{e_wf_w}.
\end{equation}
Moreover,
\begin{equation}\label{eq:fundamental-identity}
\sum_{w\mid v}e_wf_w=[E:K].
\end{equation}
Both formulas remain valid for inseparable extensions.
\end{proposition}

\begin{proof}
We have $\ord_w(a)=e_w\ord_v(a)$ and $[\kappa(w):k]=f_w[\kappa(v):k]$, proving \eqref{eq:ef-scaling}.  The second identity is the degree formula for a finite morphism of regular curves \cite[Chap~II, Prop~6.9]{hartshorne2013algebraic}.
\end{proof}

If a local quantity is defined by finite sums, minima and maxima of functions $-\log|f|_v$, then after pullback to $w$ it scales by $e_wf_w$. This applies to local Weil functions, boundary heights, escape rates, resultant terms, and the Arakelov--Green functions used in the rest of this article.

For a symmetric ample line bundle $L$ on an abelian variety $A/K$, Tate's limit
\[
 \hhat_{A,L,K}(x)= \lim_{m\to\infty}m^{-2}h_{L,K}([m]x)
\]
defines the N\'eron--Tate height.  It is nonnegative and vanishes on torsion. The construction and its local decomposition are valid over product-formula fields \cite[Chap 9]{bombieri2007heights}.  Nonnegativity can also be seen directly by choosing a projective height $h_L\geq0$ and taking the limit.

\begin{definition}
Let $A/K$ be a principally polarized abelian variety of dimension $g$. Choose a finite extension $E/K$ over which $A_E$ is semistable and has level-$\ell_p$  structure. Let $B_E$ be the normalization of $B$ inside $E$. By properness, its moduli map extends to a morphism \(p_A:B_E\to\mathcal A^*_{g,\ell_p}\) to the minimal compactification.  If $\omega$ denotes the Hodge line bundle on that compactification, define 
\begin{equation}\label{eq:stable-hodge-moduli}
 h_{\omega,K}(A)=\frac{1}{[E:K]}\deg_k p_A^*\omega.
\end{equation}
\end{definition}

This is the stable Hodge-height normalization used by Looper--Yap \cite[Sec~2.2]{looper2026uniform}. Its independence of the chosen semistable level extension follows from base change for the Hodge bundle \cite[Chap~V--VI]{faltings2013degeneration}.    For every finite extension $F/K$,
\begin{equation}\label{eq:hodge-scaling}
 h_{\omega,F}(A_F)=[F:K]h_{\omega,K}(A).
\end{equation}

\subsection{Arakelov--Green function}\label{sec:local-energy}
Let $K_0$ be an arbitrary field endowed with an absolute value (possibly trivial), let $A/K_0$ be an abelian variety, let $L$ be symmetric and ample, and let $[m]$, $m\geq2$, be the polarized endomorphism.  After replacing $L$ by a fixed power, choose an embedding $A\hookrightarrow\P^N$ for which $[m]$ extends to a morphism $f:\P^N\to\P^N$, and choose a homogeneous lift $F=(F_0,\ldots,F_N):\mathbb A^{N+1}\to\mathbb A^{N+1}$.  Thus the $F_i$ are homogeneous of the common degree $d=m^2$ and have no common nonzero zero.  For $P\in\mathbb A^{N+1}\setminus\{0\}$, put
\begin{equation}\label{eq:escape-rate-definition}
 \widehat H_F(P)=\lim_{s\to\infty}d^{-s}\log\|F^s(P)\|,  \qquad \|P\|=\max_i|P_i|,
\end{equation}
and set
\begin{equation}\label{eq:resultant-term}
r(F)=\frac{1}{d^N(N+1)(d-1)} \log|\operatorname{Res}(F)|^{-1}.
\end{equation}
The first quantity is the homogeneous escape rate and the second is the normalized Macaulay-resultant term used by Looper \cite[Sec~3.1]{looper2026uniform}. Note that for \(a\ne0\),
\[
\operatorname{Res}(aF)=a^{(N+1)d^N}\operatorname{Res}(F),  \qquad  \widehat H_{aF}=\widehat H_F+\frac{\log|a|}{d-1},
\]
and hence
\[
  r(aF)=r(F)-\frac{\log|a|}{d-1}.
\]
Thus \eqref{eq:resultant-term} is exactly the normalization for which a scalar change of the homogeneous lift cancels the corresponding change in the escape rate. Put $c(n)=h^0(A,L^{\otimes n})$.  For a basis $\cB_n=(\eta_1,\ldots,\eta_{c(n)})$ and a tuple $Q=(Q_1,\ldots,Q_{c(n)})$, Looper's Arakelov--Green function is
\begin{equation}\label{eq:green}
 g_{\cB_n}(Q)=\frac1{c(n)}\sum_i\widehat H_F(\widetilde Q_i)  -\frac1{nc(n)}\log\left| \det(\eta_j(\widetilde Q_i))_{i,j}\right|+r(F).
\end{equation}
Here $\widetilde Q_i\in\mathbb A^{N+1}\setminus\{0\}$ is any homogeneous coordinate vector representing $Q_i$, and each $\eta_j$ is evaluated on that vector as a homogeneous polynomial of degree $n$.  The determinant is the determinant of the resulting $c(n)\times c(n)$ evaluation matrix.  Formula \eqref{eq:green} is independent of the chosen homogeneous lifts and of the lift $F$ \cite[Sec 2]{looper2024arakelov}.  A change of basis adds a constant.

A \emph{good basis}  means a basis produced by Looper's explicit good-basis construction. It is selected from the finite spanning family formed from coordinate monomials and bounded products of coordinate polynomials of iterates of $F$ \cite[Sec~2.2]{looper2024arakelov}.  

\begin{proposition}\label{prop:green-field-hypotheses}\cite[Sec~2.1]{looper2024arakelov}
Over every valued field $K_0$, in arbitrary characteristic, the escape-rate limit exists and \eqref{eq:green} is independent of the homogeneous lifts of the points and of the lift $F$.  
\end{proposition}

The canonical skeleton $S(A)$ and the retraction $r_A:A^{\an}\to S(A)$ arise from Raynaud uniformization \cite{bosch1984stable,berkovich1990spectral}.  The canonical Chambert--Loir probability measure $\mu_{A,L}$ \cite{chambert2003mesures} is supported on the skeleton, and its pushforward to the tropical torus is Haar measure.  In particular it has full support on $S(A)$ \cite[Thm 1.1]{gubler2010non}, see also \cite{gubler2023monge}. For the remainder of this section, let $C$ be an algebraically closed, complete, nontrivially valued non-Archimedean field of arbitrary characteristic and base change $A$ and the dynamical data to $C$.  Define
\[
 m_n(A)=\inf_{Q\in A(C)^{c(n)}}g_{\cB_n}(Q).
\]
For a closed $Z\subset S(A)$, let
\[
  m_n(A;Z) =\inf_{\substack{Q=(Q_1,\ldots,Q_{c(n)})\\ Q_i\in r_A^{-1}(Z)\ }}  g_{\mathcal B_n}(Q).
\]
The excess $m_n(A;Z)-m_n(A)$ is independent of the basis.

\begin{theorem}\label{thm:local-gap}
Suppose $A$ has bad reduction and $\varnothing\ne Z\subsetneq S(A)$ is closed.  For every sequence of good bases as defined above,
\begin{equation}\label{eq:local-gap}
 \gamma(A,Z):=\liminf_{n\to\infty} \bigl(m_n(A;Z)-m_n(A)\bigr)>0.
\end{equation}
Moreover, for the homogeneous lift used to construct those good bases,
\begin{equation}\label{eq:local-resultant-gap}
 \gamma_F(A,Z):=\liminf_{n\to\infty} \bigl(m_n(A;Z)-r(F)\bigr)>0.
\end{equation}
\end{theorem}

\begin{proof}
Assume the liminf is zero.  Passing to a subsequence, choose tuples $Q_n=(Q_{n,1},\ldots,Q_{n,c(n)})$ with all $r_A(Q_{n,i})\in Z$ and
\[
 g_{\cB_n}(Q_n)-m_n(A)\longrightarrow0.
\]
Thus $Q_n$ is asymptotically Fekete.  The Fekete equidistribution theorem of Boucksom--Gubler--Martin, which is valid over an arbitrary complete non-Archimedean field, implies
\[ \nu_n:=\frac1{c(n)}\sum_i\delta_{Q_{n,i}} \longrightarrow\ \mu_{A,L}
\]
on $A^{\an}$ \cite[Cor D]{boucksom2022differentiability}. The set $r_A^{-1}(Z)$ is closed because $r_A$ is continuous.  Every $\nu_n$ is supported on $r_A^{-1}(Z)$, so the Portmanteau theorem for weak convergence of Radon measures shows that $\mu_{A,L}$ is also supported on $Z$ \cite[Thm~2.1]{billingsley1999convergence}. On the other hand, choose $s\in S(A)\setminus Z$.  Since $Z$ is closed and the skeleton is a compact real torus, there is an open neighborhood $U$ of $s$ in $S(A)$ disjoint from $Z$.  Then $r_A^{-1}(U)$ is an open neighborhood of $s$ in $A^{\mathrm{an}}$, disjoint from $r_A^{-1}(Z)$.  Full support on the skeleton gives $\mu_{A,L}(r_A^{-1}(U))>0$, contradicting the asserted support. Therefore the liminf in \eqref{eq:local-gap} is positive.

Looper's determinant estimate \cite[Thm~3.1]{looper2024arakelov} gives
\[
 \liminf_{n\to\infty}\bigl(m_n(A)-r(F)\bigr)\geq0.
\]
Adding this inequality to \eqref{eq:local-gap} proves
\eqref{eq:local-resultant-gap}. 
\end{proof}

\begin{remark}\label{rem:why-bad}
If the absolute value of $C$ is replaced by $|\cdot|^a$ for $a>0$, every term in \eqref{eq:green} is multiplied by $a$, and hence
\begin{equation}\label{eq:gap-scaling} \gamma(A_{|\cdot|^a},Z)=a\gamma(A_{|\cdot|},Z), \qquad \gamma_F(A_{|\cdot|^a},Z)=a\gamma_F(A_{|\cdot|},Z).
\end{equation}
In particular, at an extension $w\mid v$ the local gap is multiplied by $e_wf_w$. 
\end{remark}

For the rest of this section fix $K=k(B)$ and a polarized abelian variety $(A,L)/K$ together with a dynamical presentation as in the beginning of this section.  For each $v\in|B|$, fix an algebraic closure of $K_v$ and write $C_v=\widehat{\overline{K_v}}$.  

\begin{proposition}\label{prop:adelic-identities}
There are ordered good \(K\)-bases $\mathcal B_n$, and for each fixed polarized dynamical presentation of $A$, a constant $C_A>0$ such that every finite extension $E/K$ satisfies
\begin{equation}\label{eq:adelic-determinant}
\frac{1}{[E:K]}\sum_{w\in |B_E|}m_{n,w}(A_E) \geq -C_A\frac{\log n}{n}.
\end{equation}
If $Q=(Q_1,\ldots,Q_{c(n)})\in A(E)^{c(n)}$ and no nonzero element of $H^0(A_E,L_E^{\otimes n})$ vanishes at every $Q_i$, then
\begin{equation}\label{eq:global-green-height}
\frac{1}{c(n)}\sum_{i=1}^{c(n)}\widehat h_{A,L,K}(Q_i) =\frac{1}{[E:K]}\sum_{w\in |B_E|}g_{\mathcal B_n,w}(Q).
\end{equation}
\end{proposition}

\begin{proof}
The first assertion is Looper's $O(\log n/n)$ adelic determinant estimate \cite[Thm~1.1]{looper2024arakelov}.  For the second, the determinant and resultant sums vanish by the product formula, while the escape-rate sum is the canonical height; see \cite[Sec~2--3]{looper2024arakelov} and \cite[Cor~9.5.14]{bombieri2007heights}.
\end{proof}

\begin{lemma}\label{lem:evaluation-selection}
Let $X$ be an integral projective variety over $K$, let $L$ be ample, and let
\[
  V=H^0(X_{\overline K},L_{\overline K}^{\otimes n}),  \qquad \dim_{\overline K}V=c(n).
\]
If a finite set $P\subset X(\overline K)$ with $|P|\geq c(n)$ is not contained in the zero divisor of any nonzero element of $V$, then $P$ contains a subset $Q$ of cardinality $c(n)$ whose evaluation determinant is nonzero.
\end{lemma}

\begin{proof}
Evaluation at $x\in P$ gives a linear functional $\ell_x\in V^*$, well defined up to a nonzero scalar.  If the $\ell_x$ did not span $V^*$, their common annihilator would contain a nonzero section vanishing on $P$.  Hence they span $V^*$, and one can select a basis from them.  The corresponding subset $Q$ has nonzero evaluation determinant.
\end{proof}

\section{Dynamical charts and ultralimits}
\label{sec:uniform-package}

Let 
\begin{equation}\label{eq:level}
\ell_p=\begin{cases}
    3,&p\ne 3,\\
    4,&p=3.
\end{cases}
\end{equation}
Thus \(\ell_p\ge 3\) and \((\ell_p,p)=1\).  Fix a primitive \(\ell_p\)-th root of unity and put
\[
\kappa_p=\begin{cases}
    \mathbb Q(\zeta_{\ell_p}),&p=1,\\
    \mathbb F_p(\zeta_{\ell_p}),&p>1.
  \end{cases}
\]
The fine level-\(\ell_p\) moduli scheme and the chosen toroidal data descend over this cyclotomic field extension.  We give \(\kappa_p\) the trivial absolute value.

Let \(U\) be a quasi-projective \(\kappa_p\)-scheme.  Fix a projective compactification \(U\hookrightarrow\overline U\), and fix a coherent ideal \(\mathcal I_{\partial U}\) whose vanishing set is \(\overline U\setminus U\).  Let \(F/\kappa_p\) be a  valued extension with valuation ring \(F^\circ\), and let \(x\in U(F)\).  Properness and separatedness give a unique extension
\[
  \widetilde x:\Spec F^\circ \longrightarrow\overline U.
\]
If the closed point of \(\widetilde x\) lies in an affine chart \(V\) on which \(\mathcal I_{\partial U}\) is generated by \(f_1,\ldots,f_s\), define the
\emph{model boundary function}
\begin{equation}\label{eq:model-boundary}
\lambda_{\partial U}(x) =-\log\max_i|f_i(\widetilde x)| =\min_i\{-\log|f_i(\widetilde x)|\}.
\end{equation}
Every finitely generated ideal of a valuation ring is generated by an element of maximal norm. The value $\lambda_{\partial U}(x)$  is exactly the valuation of the finitely generated ideal \(\widetilde x^*\mathcal I_{\partial U}\subset F^\circ\). Thus, it is independent of the chart and generators.

For a product-formula field \(K=k(B)\) and \(x\in U(K)\), use this fixed model at every place and put
\[
  h_{\partial U,K}(x)=\sum_{v\in|B|}\lambda_{\partial U,v}(x).
\]
Only finitely many terms are nonzero. Replacing \(|\cdot|\) by \(|\cdot|^c\) multiplies the model boundary function by
\(c\).

\begin{definition}\label{def:ultralimit-ring}
Let \(C_j\) be algebraically closed, complete non-Archimedean fields containing \(\kappa_p\) trivially. Let \(\epsilon_j>0\), and fix a nonprincipal ultrafilter \(\omega\) on \(\mathbb N\).  Set
\[
  \mathcal U^\epsilon
  =\left\{(a_j)\in\prod_jC_j:\sup_j|a_j|^{\epsilon_j}<\infty\right\}.
\]
The seminorm
\(
  |(a_j)|_\omega= \lim_\omega |a_j|^{\epsilon_j}
\)
has kernel \(\mathfrak n_\omega\).  Write
\[
  H(\omega)=\mathcal U^\epsilon/\mathfrak n_\omega,
  \qquad
  R(\omega)=\{a\in H(\omega):|a|_\omega\le1\}.
\]

\begin{remark}
Let $J$ be an $\omega$-large set. Then we may similarly define $\mathcal U^\epsilon_J :=\{(a_j)\in \prod_{j\in J}C_j : \sup_j|a_j|^{\epsilon_j}<\infty\}$, and $\fm_\omega$ to be the kernel of the norm $\mid .\mid_\omega$. Since $J$ is $\omega$-large, the quotient $\cU^\epsilon_J/ \fm_\omega$ is identical to $H(\omega)$. 
\end{remark}

For a \(\kappa_p\)-scheme \(X\) and points \(x_j\in X(C_j)\), an
\((\epsilon_j)\)-\emph{degeneration} is a point
\(x\in X(\mathcal U^\epsilon)\) inducing every \(x_j\).  Its base change
\(x_\omega\in X(H(\omega))\) is the \emph{fundamental ultralimit}.
\end{definition}

\begin{lemma}
The kernel \(\mathfrak n_\omega\) is a maximal ideal of \(\mathcal U^\epsilon\)
\end{lemma}
\begin{proof}
Let \(b=(b_j)\notin\mathfrak n_\omega\), and write \(c=\lim_\omega |b_j|^{\epsilon_j}>0\).  On an \(\omega\)-large set, \(|b_j|^{\epsilon_j}\ge c/2\).  Let $a_j = b_j^{-1}$ on this set and let \(a_j\) be  \(1\) on the complement. The new sequence $a=(a_j)$ satisfies
\[
\sup_{j} |a_j|^{\epsilon_j} \le \max \{1, \frac 2 c \}
\]
Therefore, $a$ belongs to \(\mathcal U^\epsilon\) and $ab =1$ modulo $\fn_\omega$.  Thus every nonzero class is a unit, proving maximality.
\end{proof}

\begin{lemma}\label{lem:ultralimit-field}
The quotient \(H(\omega)\), with its quotient
norm, is a complete algebraically closed non-Archimedean valued field.
\end{lemma}

\begin{proof}
Let \((x_m)\) be Cauchy in \(H(\omega)\).  Take  a subsequence \((y_m)\) such that \(|y_{m+1}-y_m|_\omega<2^{-m-2}\), and choose representatives \(a^{(m)}=(a^{(m)}_j)\in\cU^\epsilon\) of $y_m$. By the definition of the seminorm, there are decreasing sets \(J_m\in\omega\) such that, whenever \(j\in J_m\), one has
\[
 |a^{(k+1)}_j-a^{(k)}_j|^{\epsilon_j}\le 2^{-k-1}
 \quad(1\le k\le m).
\]
Let $s(j)=1$ if $j\not \in J_1$ and in the complement let $s(j)$ be the maximum index at most $j$ such that $j\in J_{s(j)}$. Define a new sequence $a=(a_j)$ by \(a_j=a^{(s(j))}_j\).
Note that \(s(j)\to_\omega\infty\) and for any $j$ we have 
\begin{equation}\label{eq:ultralimit-diagonal-bound}
|a_j|^{\epsilon_j} \le  \max\left\{   |a^{(1)}_j|^{\epsilon_j},    \max_{1\le k<s(j)} |a^{(k+1)}_j-a^{(k)}_j|^{\epsilon_j} \right\} \le\max\{|a^{(1)}_j|^{\epsilon_j},1/4\}.
\end{equation}
Hence
\[
\sup_j|a_j|^{\epsilon_j} \le\max\left\{\sup_j|a^{(1)}_j|^{\epsilon_j},1/4\right\}<\infty,
\]
so \(a=(a_j)\) belongs to \(\mathcal U^\epsilon\). For fixed \(m\), on the \(\omega\)-large set \(J_m\cap\{j\ge m\}\) one has \(s(j)\ge m\), and 
\[
 |a_j-a^{(m)}_j|^{\epsilon_j}
 \le\max_{m\le k<s(j)}2^{-k-1}
 \le2^{-m-1}.
\]
Thus $y_m$ converges to $a$ and  the original Cauchy sequence has a limit. 

Let $P(T)\in H(\omega)[T]$ be a polynomial given by $P(T) = T^q+\sum_{i=0}^{q-1}a_{i}T^i$. Write $a_i = (a_{i,j})$ and $P_j=T^q+\sum_{i<q}a_{i,j}T^i$. Then, $P_j$ is a polynomial with coefficients in $C_j$ and has a root $z_j$.  The non-Archimedean root bound
\[
 |z_j|\le \max\{1,|a_{0,j}|,\ldots,|a_{q-1,j}|\}
\]
for \(P_j\) shows that
\((z_j)\in\mathcal U^\epsilon\).  Its class is a root of the original
polynomial $P(T)$.  Thus \(H(\omega)\) is algebraically closed.  
\end{proof}

\begin{lemma}\label{lem:boundary-degeneration}
Let \(X\) be a separated quasi-projective \(\kappa_p\)-scheme equipped with the model boundary function as above, and let \(x_j\in X(C_j)\).  An \((\epsilon_j)\)-degeneration exists if and only if
\begin{equation}\label{eq:moderate-boundary}
\sup_j\epsilon_j\lambda_{\partial X}(x_j)<\infty,
\end{equation}
and it is unique when it exists. If \(\lambda_{\partial X}(x_j)>0\) and \(\epsilon_j=\lambda_{\partial X}(x_j)^{-1}\), then the degeneration exists and $\lambda_{\partial X}(x_\omega)=1$. Moreover, \(x_\omega\) does not extend to an \(X\)-point over \(R(\omega)\). 
\end{lemma}
\begin{proof}
Write \(D\) for the boundary closed subscheme defined by the ideal sheaf \(\mathcal I_{\partial X}\).  Choose a finite affine cover \(\overline X= \bigcup_{\alpha=1}^mV_\alpha,\) where \( V_\alpha=\Spec A_\alpha\), and generators \(\mathcal I_{\partial X}|_{V_\alpha} =(f_{\alpha,1},\ldots,f_{\alpha,s_\alpha})\). Let \(R_j=C_j^\circ\).  Properness gives a unique extension \(\widetilde x_j:\Spec R_j\to\overline X\).  Choose \(\alpha(j)\) so that the closed point lies in \(V_{\alpha(j)}\).  Define $J_\alpha:=\{j:\alpha(j)=\alpha\}$ and  $e_\alpha=(\mathbf 1_{J_\alpha}(j))_j$ where $\mathbf 1_{J_\alpha}$ is the characteristic function of $J_\alpha$. Note that $e_\alpha\in\mathcal U^\epsilon$. We have
\begin{equation}  \label{eq:chart-idempotents}
e_\alpha^2=e_\alpha,  \qquad e_\alpha e_\beta=0\quad(\alpha\ne\beta),  \qquad \sum_\alpha e_\alpha=1.
\end{equation}   
Consequently, $\mathcal U^\epsilon\simeq\prod_\alpha e_\alpha\mathcal U^\epsilon$, and $\Spec\mathcal U^\epsilon =\coprod_\alpha\Spec(e_\alpha\mathcal U^\epsilon)$. For \(a\in A_\alpha\), define
\[
\Phi_\alpha(a)_j=\begin{cases}
a(\widetilde x_j), & j\in J_\alpha,\\
0, & j\notin J_\alpha.
\end{cases}
\]
For \(j\in J_\alpha\), the value belongs to \(R_j\), so
\(| \Phi_\alpha(a)_j|^{\epsilon_j} \le1\).  Hence \(\Phi_\alpha: A_\alpha\to e_\alpha\mathcal U^\epsilon\) is a homomorphism.  The maps \(\Phi_\alpha\) glue across the clopen decomposition to a point \(\overline x\in\overline X(\mathcal U^\epsilon)\) with fibers \(x_j\). We still need to show that $\ol x$ factors through $X\hookrightarrow \ol X$. Refine the partition by 
\[
J_{\alpha r}=\left\{j\in J_\alpha: r=\min\operatorname*{arg\,max}_{1\le i\le s_\alpha} |f_{\alpha,i}(\widetilde x_j)|\right\}
\]
where $\operatorname*{arg\,max}$ is the set of  indices where the maximum is achieved and let \(e_{\alpha r} = (\mathbf 1_{J_{\alpha r}}(j))_j\). Then
\[
e_{\alpha r}e_{\alpha t}=0\quad(r\ne t), \qquad \sum_re_{\alpha r}=e_\alpha.
\]
Define \(b_\alpha=\sum_{r=1}^{s_\alpha} e_{\alpha r}\Phi_\alpha(f_{\alpha,r}) \in e_\alpha\,\mathcal U^\epsilon\). For \(j\in J_\alpha\), one has the exact computation
\begin{equation}\label{eq:b-coordinate}
|(b_\alpha)_j| =\max_i|f_{\alpha,i}(\widetilde x_j)| =\exp\bigl(-\lambda_{\partial X}(x_j)\bigr).
\end{equation}
Moreover, if
\[
(q_{\alpha i})_j=\begin{cases}
f_{\alpha,i}(\widetilde x_j)/(b_\alpha)_j,&j\in J_\alpha,\\ 0,&j\notin J_\alpha,
\end{cases}
\]
then \(|(q_{\alpha i})_j|\le1\), so \(q_{\alpha i}\in e_\alpha \, \mathcal U^\epsilon\) and \( \Phi_\alpha(f_{\alpha,i})=b_\alpha q_{\alpha i}\). Conversely, the displayed definition of \(b_\alpha\) writes it as a linear combination of the \(\Phi_\alpha(f_{\alpha,i})\).  Thus
\[
\bigl(\Phi_\alpha(f_{\alpha,1}),\cdots, \Phi_\alpha(f_{\alpha,s_\alpha})\bigr)=(b_\alpha)
\]
inside \(e_\alpha\,\mathcal U^\epsilon\).  With \(b=\sum_\alpha b_\alpha\), we have therefore proved $\overline x^*\mathcal I_{\partial X}=(b)$. An element \(c=(c_j)\in\mathcal U^\epsilon\) is a unit if and only if $\inf_j|c_j|^{\epsilon_j}>0$. Equations \eqref{eq:b-coordinate} give
\[
\inf_j|b_j|^{\epsilon_j} =\exp\left(-\sup_j\epsilon_j\lambda_{\partial X}(x_j)\right).
\]
Hence \(b\) is a unit exactly when \eqref{eq:moderate-boundary} holds.  The map \(\ol x : \Spec\, \mathcal U^\epsilon\to \overline X\) factors through \(X=\overline X\setminus D\) exactly when the pulled-back boundary ideal is the unit ideal.  This proves existence and necessity. Uniqueness follows from \cite[Prop 4.1]{looper2026uniform}. Even though they work over characteristic $0$, their proof for this proposition does not use this property and holds over any characteristic. 

Now assume \(\lambda_j =\epsilon_j^{-1}\), then the degeneration exists. A finite partition has exactly one \(\omega\)-large member.  Choose \(\alpha_0\) with \(J_{\alpha_0}\in\omega\) and then \(r_0\) with \(J_{\alpha_0r_0}\in\omega\).  In \(H(\omega)\), the class of a characteristic idempotent is \(1\) when its set is \(\omega\)-large and \(0\) otherwise. Hence $e_{\alpha_0, r_0} =1$ and $e_{\alpha, r} =0$ if $(\alpha,r)\neq (\alpha_0, r_0)$. Every \(\Phi_{\alpha_0}(a)\) has rescaled norm at most \(1\), so its class defines an \(R(\omega)\)-point \(\overline x_R:\Spec R(\omega)\to V_{\alpha_0}\) with generic fiber \(x_\omega\).  On the \(\omega\)-large set \(J_{\alpha_0}\),
\[
\left(\max_i|f_{\alpha_0,i}(\widetilde x_j)|\right)^{\epsilon_j}  =\exp(-\epsilon_j\lambda_j)=e^{-1}.
\]
It follows that
\(\max_i|f_{\alpha_0,i}(\overline x_R)|_\omega=e^{-1}\),
with equality attained by \(i=r_0\). This proves $\lambda_{\partial X}(x_\omega)=1$. Finally, for arbitrary \(c_i\in R(\omega)\),
\[
\left|\sum_i c_i f_{\alpha_0,i}(\overline x_R)\right|_\omega \le \max_i|c_i|_\omega |f_{\alpha_0,i}(\overline x_R)|_\omega   \le e^{-1}<1.
\]
Thus the pulled-back boundary ideal is not the unit ideal, so
\(\overline x_R\) does not factor through \(X\). 
\end{proof}

\begin{lemma}\label{lem:bad-ultralimit}
Let \(X\) be the fine moduli scheme of principally polarized abelian varieties with full level \(\ell_p\) structure, and let \(\overline X\) be a projective toroidal compactification carrying the universal semiabelian scheme. Take $x_j\in X(C_j)$ where $\lambda_{\partial X}(x_j)>0$. Under the fundamental rescaling $\epsilon_j = \lambda_{\partial X}(x_j)^{-1}$, the induced abelian variety \(A_\omega/H(\omega)\), corresponding to $x_\omega$, has bad reduction and admits a semiabelian model over
\(R(\omega)\).
\end{lemma}

\begin{proof}
If \(A_\omega\) had good reduction, it would extend to an abelian scheme over \(R(\omega)\). Because \(\ell_p\) is invertible, the \(\ell_p\)-torsion is finite \'etale and the generic full level structure extends uniquely over the valuation ring. The resulting moduli map \(\Spec R(\omega)\to X\) would extend the point \(x_\omega\), contradicting Lemma~\ref{lem:boundary-degeneration}.  Thus the reduction is bad.  Since \(\overline X\) is proper, the map \(x_\omega\) extends to \(\Spec R(\omega)\to\overline X\). Pulling back the universal semiabelian scheme gives the asserted model.  See \cite[Chap~IV--V]{faltings2013degeneration}.
\end{proof}

\begin{lemma}
\label{lem:uniform-escape}
Fix \(N\) and \(d\ge2\).  There is a constant \(C_{N,d}>0\) such that,
for every non-Archimedean valued field, every degree-\(d\) homogeneous lift
\(F=(F_0,\ldots,F_N)\) with a coefficient equal to \(1\), resultant
\(R\ne0\), every \(P\ne0\), and every \(s\ge0\),
\begin{equation}\label{eq:uniform-escape}
\left| \widehat H_F(P)-d^{-s}\log\|F^s(P)\| \right|
 \le
\frac{C_{N,d}}{d^s(d-1)} \left(\log\|F\|+\log^+|R|^{-1}\right).
\end{equation}
\end{lemma}

\begin{proof}
We have \(\|F(P)\|\le\|F\|\,\|P\|^d\).  The universal Macaulay B\'ezout identity provides integers \(\alpha,\beta,D>0\), depending only on \(N,d\), and homogeneous polynomials \(G_{ij}\), with integral coefficients polynomial in the coefficients of \(F\), such that
\[
R^\alpha X_i^D=\sum_{j=0}^NG_{ij}(F;X)F_j(X), \qquad \deg_XG_{ij}=D-d,
\]
and the coefficient degree in \(F\) is at most \(\beta\).  Choose \(i\) with \(|P_i|=\|P\|\) and evaluate the $i$-th identity at $P$.  Since every integer has non-Archimedean norm at most \(1\), we get
\[
\|F(P)\|\ge |R|^\alpha\|F\|^{-\beta}\|P\|^d.
\]
Because one coefficient of \(F\) is \(1\), \(\|F\|\ge1\). Moreover the resultant is a fixed integral polynomial in the coefficients, so \(\log^+|R|\ll_{N,d}\log\|F\|\).  Therefore there exists a constant $C_{N,d}$ such that
\[
 \left|\log\|F(P)\|-d\log\|P\|\right|
 \le C_{N,d}\left(\log\|F\|+\log^+|R|^{-1}\right).
\]
For any $t>s$ we have that 
\begin{align*}
\left|d^{-t}\log\|F^t(P)\| -d^{-s}\log\|F^s(P)\|\right| &= \left|\sum_{i=s+1}^t d^{-i }\log\|F^i(P)\| - d^{-(i-1)}\log\|F^{(i-1)}(P)\| \right|  \\&\le  \sum_{i=s+1}^t\frac{C_{N,d}}{d^i} \left(\log\|F\|+\log^+|R|^{-1}\right)\\& \le \frac{C_{N,d}}{d^s(d-1)} \left(\log\|F\|+\log^+|R|^{-1}\right)
\end{align*}
By taking the limit as $t\to \infty$, we get inequality \eqref{eq:uniform-escape}.
\end{proof}

Denote $A:=\mathcal U^\epsilon$. Let \(\rho_j:A\to C_j\) be the \(j\)-th coordinate map and let \(\rho_\omega:A\to H(\omega)\) be the natural quotient map.

\begin{lemma}\label{lem:projective-ultralimit}
 Let \(Y\hookrightarrow\mathbb P^N_A\) be a closed subscheme and write \(Y_j=Y\times_{A,\rho_j}C_j\).   Every sequence \(y_j\in Y_j(C_j)\) determines a unique point \(\widetilde y\in Y(A)\) whose base change along \(\rho_j\) is \(y_j\).  Moreover, $\widetilde{y}$ determines an ultralimit \(y_\omega\in Y(H(\omega))\).  The construction is functorial for morphisms of projective \(A\)-schemes. 
\end{lemma}

\begin{proof}
Let $y_j\in Y_j(C_j)$ and choose homogeneous coordinates \(y_j=[a_{j,0}:\cdots:a_{j,N}]\) with \(\max_i|a_{j,i}|=1\).  Let \(\iota(j)\) be the least index at which the maximum is attained. Let \(J_i=\{j:\iota(j)=i\}\), and let \(e_i=(\mathbf 1_{J_i}(j))_j\in A\).  Then
\begin{equation}  \label{eq:projective-idempotents}
e_i^2=e_i, \qquad e_ie_k=0\quad(i\ne k), \qquad \sum_{i=0}^Ne_i=1.
\end{equation}
Define \(b_i=(b_{i,j})_j\) by
\[
b_{i,j}=\begin{cases}
 a_{j,i}^{-1},&j\in J_i,\\0,&j\notin J_i.
\end{cases}
\]
Since \(|a_{j,i}|=1\) on \(J_i\), one has
\(\sup_j|b_{i,j}|^{\epsilon_j}\le1\), so \(b_i\in A\).  If
\(a_i=(a_{j,i})_j\), then
\begin{equation}\label{eq:projective-unit-ideal}
b_i a_i=e_i, \qquad \sum_{i=0}^N b_i a_i=\sum_{i=0}^Ne_i=1.
\end{equation}
Thus the \(a_i\) generate the unit ideal and define an \(A\)-point of projective space. Let \(f=\sum_\nu c_\nu X^\nu\) be a homogeneous equation of \(Y\), where \(c_\nu=(c_{\nu,j})_j\in A\).  Since ring operations in \(A\) are coordinatewise, we get
\[
f(a_0,\ldots,a_N)=\bigl(f_j(a_{j,0},\ldots,a_{j,N})\bigr)_j=0.
\]
Hence the point $\widetilde y := [a_0:\cdots:a_N]$ factors through \(Y\).  Base change along \(A\to H(\omega)\) gives \(y_\omega\).  Exactly one of \(J_0,\cdots J_N\)  is \(\omega\)-large, let $i_0$ be the index for which $J_{i_0}$ is \(\omega\)-large. The reduced coordinate \(\bar a_{i_0}\) has norm \(1\), so the homogeneous tuple remains a valid point of $\P^N_{H(\omega)}$ after reduction.

Note that \(\bigcap_j\ker\rho_j=0\) and \(\coprod_j\Spec C_j\to\Spec A\) is schematically dense.  Since every projective \(A\)-scheme is separated, two \(A\)-points with the same \(C_j\)-fibers are equal.  Let \(\phi:Y\to Z\) be a morphism of projective \(A\)-schemes. Then \(\phi(\widetilde y)\) and the point constructed from \((\phi_j( y_j))_j\) have the same \(C_j\)-fibers and they are equal.  After base change,
\[
(\phi_j(y_j))_\omega=\phi_\omega(y_\omega).
\]
\end{proof}

\begin{lemma}\label{lem:green-continuity}
Let  \(F:\mathbb P^N_A\to\mathbb P^N_A\) be a degree-\(d\) endomorphism that is represented by a homogeneous lift over \(A\), with one coefficient normalized to \(1\). Fix a list \(\mathbf s=(s_1,\ldots,s_{c(q)})\) of relative homogeneous sections. Let \(\mathbf s_j \) and $\mathbf s_\omega$ denote the base change of $\mathbf s$ under $\rho_j$ and $\rho_\omega$, respectively. Let \(G_{\mathbf s_\omega}\) denote the expression \eqref{eq:green} formed with $\mathbf s_\omega$, with value \(+\infty\) when its evaluation determinant vanishes. Assume that $\mathbf s_j$ is a good basis on an $\omega$-large set, and write $B_{q,j} = \mathbf s_j$. If \(\mathbf s_\omega\) is a basis, this is the usual good-basis Green function.  For tuples \(Q_j\) of cardinality $c(q)$, and their ultralimit \(Q_\omega\), one has, in the extended real line,
\begin{equation}\label{eq:green-continuity}
\lim_\omega\epsilon_j g_{\mathcal B_{q,j}}(Q_j) =G_{\mathbf s_\omega}(Q_\omega).
\end{equation}
Moreover,
\begin{equation}\label{eq:resultant-continuity}
  \lim_\omega\epsilon_j r(F_j)=r(F_\omega).
\end{equation}
\end{lemma}
\begin{proof}
Since $F$ is an endomorphism of all of \(\mathbb P^N_A\), its Macaulay resultant \(R=\operatorname{Res}(F)\) is a unit of \(A\). Denote by $F_j$ the base change of $F$ along $\rho_j$ and let $R_j$ be its Macaulay resultant. We normalize the homogeneous coordinates of every entry of \(Q_j\) to have norm one, as in the proof of Lemma~\ref{lem:projective-ultralimit} and we use $P_j$ for points in the $c(q)$-tuple $Q_j$ and $P_\omega$ for points of $Q_\omega$.  For every fixed polynomial expression \(\cP\) over \(A\),
\begin{equation}\label{eq:polynomial-continuity}
|\cP_\omega(\bar a)|_\omega =\lim_\omega|\cP_j(a_j)|^{\epsilon_j}.
\end{equation}
Hence, $|\operatorname{Res}(F_\omega)|_\omega=\lim_\omega|\operatorname{Res}(F_j)|^{\epsilon_j}$, and by definition, we get $ \lim_\omega\epsilon_j r(F_j)=r(F_\omega)$. For every fixed iterate \(s\), 
\[
\lim_\omega\epsilon_j\log\|F_j^s(P_j)\|=\log\|F_\omega^s(P_\omega)\|_\omega.
\]
Put $M=\sup_j\|F_j\|^{\epsilon_j}$ and $\delta_R=\inf_j|R_j|^{\epsilon_j}$. Note that $M<\infty$ and since $R$ is a unit in $A$ we have that $\delta_R>0$. The uniform estimate in Lemma~\ref{lem:uniform-escape} gives
\[
\epsilon_j\left| \widehat H_{F_j}(P_j)-d^{-s}\log\|F_j^s(P_j)\| \right| \le \frac{C_{N,d}}{d^s(d-1)} \left(\log M+\log^+\delta_R^{-1}\right).
\]
The morphism \(F_\omega\) has coefficient norm at most \(M\), resultant norm at least \(\delta_R\), and still has a coefficient \(1\). Hence the same lemma gives 
\[
\left| \widehat H_{F_\omega}(P_\omega)-d^{-s}\log\|F_\omega^s(P_\omega)\| \right| \le \frac{C_{N,d}}{d^s(d-1)} \left(\log M+\log^+\delta_R^{-1}\right).
\]
Taking the \(\omega\)-limit at fixed \(s\) and then letting \(s\to\infty\) proves 
\[
\lim_\omega\epsilon_j\widehat H_{F_j}(P_j)=\widehat H_{F_\omega}(P_\omega).
\]
If \(D_j\) is the evaluation determinant, then \(D_\omega\) is the class of \((D_j)\), and therefore
\[
\lim_\omega\epsilon_j(-\log|D_j|)=-\log|D_\omega|_\omega
\]
in the extended reals. The value is \(+\infty\) precisely when \(D_\omega=0\).  Combining these identities proves \eqref{eq:green-continuity}. 
\end{proof}

Let \(K\) be a field, let \(A/K\) be an abelian variety, and let a very ample line bundle \(L\) give a projectively normal closed immersion
\[
\iota:A\hookrightarrow\mathbb P^N_K, \qquad \iota^*\mathcal O_{\mathbb P^N_K}(1)\simeq L.
\]
This means that \( K[X_0,\ldots,X_N]_q\longrightarrow H^0(A,L^{\otimes q}) \)
is surjective for every \(q\ge1\).
Let
\[
F=(F_0,\ldots,F_N), \qquad F_i\in K[X_0,\ldots,X_N]_d, \qquad d\ge2,
\]
be a homogeneous lift with \(\operatorname{Res}_d (F_0,\ldots,F_N)\ne0\).  Thus \(F\) defines an endomorphism of \(\mathbb P^N_K\). We assume that $F$ preserves \(\iota(A)\) and extends the polarized endomorphism given by multiplication by $m\in \Z$. For \(k\ge1\), write \(F^{[k]}=(F_0^{[k]},\ldots,F_N^{[k]})\) for the \(k\)-th iterate of $F$.  Let \(b=d(N+1)\).  If \(q\ge b\), define
\[
r_-(q)=\left\lfloor \log_{(N+1)/N}\!\bigl(\max\{1,q/b\}\bigr)\right\rfloor, \qquad r_+(q)=\left\lfloor \log_{(2N+2)/(2N+1)}q \right\rfloor.
\]
Let \(\Lambda_q(N,d)\) be the finite set of tuples $\lambda= \bigl(\alpha;r;(i_\ell,k_\ell,e_\ell)_{1\le\ell\le r}\bigr)$ satisfying
\[
\begin{gathered}
 r\in\mathbb Z_{\ge0},\qquad r_-(q)\le r\le r_+(q),\qquad
 \alpha\in\mathbb Z_{\ge0}^{N+1},\qquad |\alpha|<b,\\
 i_\ell\in\{0,\ldots,N\},\qquad k_\ell\in\mathbb Z_{\ge1},\qquad
 e_\ell\in\{1,\ldots,d-1\},\\
 |\alpha|+\sum_{\ell=1}^r e_\ell d^{k_\ell}=q.
\end{gathered}
\]
Here \(r\) is the number of factors, \(\ell\) labels a factor occurrence,
and \(i_\ell\) selects the \(i_\ell\)-th coordinate polynomial of the
iterate \(F^{[k_\ell]}\) where repetitions are allowed.  Put
\[
 P_\lambda(F;X)
 =X^\alpha\prod_{\ell=1}^r
   \bigl(F_{i_\ell}^{[k_\ell]}(X)\bigr)^{e_\ell}
 \in K[X_0,\ldots,X_N]_q,
 \qquad
 X^\alpha=\prod_{i=0}^NX_i^{\alpha_i},
 \qquad |\alpha|=\sum_{i=0}^N\alpha_i,
\]
with the empty product interpreted as \(1\).  If \(1\le q<b\), instead set
\[
 \Lambda_q(N,d)
 =\{\alpha\in\mathbb Z_{\ge0}^{N+1}:|\alpha|=q\},
 \qquad P_\alpha(F;X)=X^\alpha.
\]
We refer to such $\lambda$'s as recipe indices. Define
\[
 s_\lambda=\iota^*P_\lambda(F;X)\in H^0(A,L^{\otimes q}),
 \qquad
 \mathcal S_q(A,F)=(s_\lambda)_{\lambda\in\Lambda_q(N,d)}.
\]
This is the monomial-expanded form of Looper's family
\cite[Sec~2.2]{looper2024arakelov}.  Put \(c(q)=h^0(A,L^{\otimes q})\).  An ordered
\(c(q)\)-tuple of distinct recipe
indices whose associated sections form a basis of
\(H^0(A,L^{\otimes q})\) is called a \emph{good basis}. 

\begin{lemma}\label{lem:good-basis-base-change}
For every \(q\ge1\), the finite list \(\mathcal S_q(A,F)\) spans \(H^0(A,L^{\otimes q})\), and contains a good basis.  If \(K'/K\) is any field extension, then \(\mathcal S_q(A_{K'},F_{K'})\) is the termwise scalar extension of \(\mathcal S_q(A,F)\).  An ordered \(c(q)\)-tuple of distinct recipe indices is a good basis over \(K\) if and only if the corresponding tuple is a good basis over \(K'\).  
\end{lemma}

\begin{proof}
Looper's  argument in \cite[Sec~2.2]{looper2024arakelov} shows that \(\mathcal S_q(A,F)\) spans \(H^0(A,L^{\otimes q})\) and contains a basis. Iteration, multiplication, and restriction commute with \(K\to K'\), so the indexed list commutes term by term with scalar extension.  Moreover, flat base change gives
\[
 H^0(A,L^{\otimes q})\otimes_KK' \simeq H^0(A_{K'},L_{K'}^{\otimes q}).
\]
Since \(K'/K\) is faithfully flat, the wedge of an ordered \(c(q)\)-tuple is nonzero before extension if and only if it is nonzero afterwards. The recipe type is unchanged, so the good-basis properties are preserved as well.
\end{proof}

\begin{proposition}\label{prop:extend-multiplication}
Let \(S\) be a Noetherian affine \(\kappa_p\)-scheme and let \(\pi:\mathcal A\to S\) be an abelian scheme.  Suppose that a symmetric, rigidified, relatively very ample line bundle \(\mathcal L\) gives a projectively normal closed immersion \(\mathcal A\hookrightarrow\mathbb P^N_S\), and fix \(s\in S\).  There are an integer \(M\ge2\) and a Zariski neighborhood \(V\) of \(s\) such that multiplication by \(M\) on \(\mathcal A_V\) extends to an endomorphism of \(\mathbb P^N_V\) of degree \(M^2\).
\end{proposition}

\begin{proof}

If \(\kappa(s)\) is infinite, the result follows from
\cite[Prop.~6.5]{looper2026uniform}. Assume that \(\kappa(s)\) is finite and let \(m\geq 2\). Then by Poonen's extension theorem \cite[Thm.~2.2]{poonen2013extending}, for the polarized endomorphism \([m]\colon \mathcal A_s\to\mathcal A_s\),  there is an \(r\geq 1\) such that $[m]^r=[m^r]$ extends to an endomorphism of \(\mathbb P^N_s\). Set \(M=m^r\) and \(e=M^2\), and write the corresponding degree-\(e\) tuple as $\overline G=(\overline G_0,\ldots,\overline G_N)$.

After replacing \(S\) by an affine open neighborhood \(V\) of \(s\), base change identifies the fiber of \(\pi_*\mathcal L^{\otimes e}\) at \(s\) with \(H^0(\mathcal A_s,\mathcal L_s^{\otimes e})\), and \(\pi_*\mathcal L^{\otimes e}\) is finite locally free. Let \(p\colon\mathbb P^N_V\to V\) be the projection and let
\[
\rho_e : p_*\mathcal O_{\mathbb P^N_V}(e) \longrightarrow \pi_*\mathcal L^{\otimes e}
\]
be the restriction. Let \(\mathcal C_e=\operatorname{coker}(\rho_e)\). Its fiber at \(s\) is the cokernel of
\[
H^0(\mathbb P^N_s,\mathcal O(e)) \longrightarrow H^0(\mathcal A_s,\mathcal L_s^{\otimes e}),
\]
which vanishes because $\cL$ is projectively normal. Thus \(\mathcal C_e\otimes\kappa(s)=0\), and Nakayama's lemma gives \((\mathcal C_e)_s=0\). Shrinking \(V\) again, we may assume that \(\mathcal C_e=0\) on $V$. The map \(\rho_e\) is therefore surjective. Its kernel \(\mathcal K_e\) is locally free, and the exact sequence
\[
0\to \mathcal K_e\to p_*\mathcal O_{\mathbb P^N_V}(e) \longrightarrow \pi_*\mathcal L^{\otimes e}\to 0
\]
commutes with base change. The rigidified symmetric isomorphism $[M]^*\mathcal L\simeq \mathcal L^{\otimes e}$ allows us to pull back the coordinate sections \(x_0,\ldots,x_N\) to sections $t_i\in H^0(\mathcal A_V,\mathcal L^{\otimes e})$. Since \(V\) is affine, we can lift $t_i$ to
\[
G_i^0\in H^0(\mathbb P^N_V,\mathcal O(e))
\]
Both \(\overline G_i\) and \(G_i^0|_s\) restrict to \(t_i|_s\), so
\[
\overline G_i-G_i^0|_s \in \mathcal K_e\otimes\kappa(s) =H^0(\mathbb P^N_s,\mathcal I_{\mathcal A_s}(e)).
\]
Each of these differences lifts to a section \(H_i\in H^0(V,\mathcal K_e)\). Set
\[
G_i=G_i^0+H_i.
\]
The tuple \(G=(G_0,\ldots,G_N)\) restricts to \([M]\) on \(\mathcal A_V\) and specializes at \(s\) to \(\overline G\). The Macaulay resultant of \(G\) specializes to the nonzero resultant of \(\overline G\). After shrinking \(V\) to the locus where this resultant is invertible, \(G\) defines an endomorphism of \(\mathbb P^N_V\) of degree \(e=M^2\) extending \([M]\colon\mathcal A_V\to\mathcal A_V\).
\end{proof}

\subsection{Dynamical Data}

In this subsection, we prove an analogue of \cite[Prop~6.4]{looper2026uniform}  in arbitrary characteristic. 

\begin{definition}\label{def:centered}
For \(A/C\) represented by a point of \(U_a(C)\), put
\[
  A\text{ is }C_0\text{-centered in }U_a
  \quad\Longleftrightarrow\quad
  \lambda_{\partial U_a}(A)
  \le C_0\lambda_{\partial X}(A).
\]
\end{definition}

Let $U_1,\cdots, U_r$ be a finite covering of $X_{g,\ell_p}$ over $\kappa_p$, then $\lambda_{\partial X} = \min_i \lambda_{\partial U_i}$. Hence, every abelian variety $A$ over $C$ is $1$-centered in at least one open $U_i$. 

\begin{theorem}\label{thm:dynamical-data}
For every \(g\) and characteristic exponent \(p\), there are data with the
following properties.
\begin{enumerate}[label=\textnormal{(D\arabic*)},leftmargin=3.8em]
\item A fine moduli scheme \(X=X_{g,\ell_p}\) of principally polarized \(g\)-dimensional abelian varieties with full level-\(\ell_p\) structure, a projective toroidal compactification \(\overline X\) carrying a universal semiabelian scheme, the Hodge bundle \(\omega\), and a finite affine cover \(X=\bigcup_{a=1}^tU_a\) over \(\kappa_p\).

\item The canonical symmetric relatively ample bundle obtained from the universal Poincar\'e bundle, and its fixed third power \(L\), such that on each \(U_a\) it gives a projectively normal immersion in \(\mathbb P^{N_a}_{U_a}\). For some \(m_a\ge2\), multiplication by \(m_a\) extends to an endomorphism of \(\mathbb P^{N_a}_{U_a}\) of degree \(m_a^2\), with a normalized homogeneous lift \(F_a\).  Put \(c(n)=h^0(A,L^{\otimes n})\).

\item Fix $a$. There is a constant $C_a$ with the following property. 

Given a product-formula field \(K_0\) and \(x\in U_a(K_0)\), choose  a good basis \(\mathcal B_{n,a,x}\). Let $E/K_0$ be a finite extension. The basis  \(\mathcal B_{n,a,x}\) is base-changed to \(E\) and to every completion of \(E\) at its valuations.  For any $c(n)$-tuple $Q$ whenever the evaluation determinant of \(Q\) is nonzero, there are numbers \(\eta_{n,w} \ge0\) satisfying
\begin{align}
g_{\mathcal B_{n,a,x},w}(Q)&\ge r(F_a)_w-\eta_{n,w},
\label{eq:local-determinant}\\
\frac1{[E:K_0]}\sum_{w\in|B_E|}\eta_{n,w}
&\le C_a\frac{\log n}{n}h_{\omega,K_0}(A),
\label{eq:determinant-error}\\
\sum_{w\in|B_E|}r(F_a)_w&=0,
\label{eq:resultant-product}
\end{align}

\item Let $x\in X(K)$ correspond to an abelian variety $A$, and let $v$ be a semistable place of bad reduction of $A$. Then
\begin{equation}\label{eq:boundary-order}
    \lambda_{\partial X,v}(x)\ge d_v.
\end{equation}
\end{enumerate}
All these data depend only on \(g\) and \(p\), not on \(k\), \(B\), or the
abelian variety \(A\).
\end{theorem}

We break the proof of Theorem~\ref{thm:dynamical-data} into the following Lemmas.

\begin{lemma}\label{lem:data-moduli}
The data in \textnormal{(D1)--(D2)} exist over \(\kappa_p\).
\end{lemma}
\begin{proof}
The fine moduli scheme, its universal abelian scheme, the Hodge bundle, and a projective toroidal compactification carrying the universal semiabelian degeneration are supplied by the algebraic theory of Faltings--Chai \cite[Chap~IV--V]{faltings2013degeneration}. Let $\cA$ be the universal abelian scheme over $X$ and $\lambda:\cA\to \cA^\vee$ be its principal polarization. Let 
\[\Gamma_\lambda:\cA\to \cA\times_X\cA^\vee\] 
be the graph of $\lambda$. Let $\cM:=\Gamma_\lambda^*\cP$ where $\cP$ is the universal Poincar\'e bundle. Sekiguchi's Theorem  \cite{sekiguchi1977projective} implies that $\cM^3$ is very ample and normally generated on every abelian fiber.  After taking this power once and for all, cohomology and base change make the relevant direct image locally free, and on an affine neighborhood a basis gives a projectively normal immersion \(\mathcal A\hookrightarrow\mathbb P^N_X\).   Proposition~\ref{prop:extend-multiplication} supplies open neighborhoods for any $x\in X$ with the dynamical data. Since $X$ is Noetherian  and quasi-compact, we can find a finite subcover.
\end{proof}

\begin{lemma}\label{lem:fixed-model-height}
Let \(X\) be an integral projective variety over a  field \(\kappa\), and let \(\Omega\) be an ample line bundle on \(X\).
\begin{enumerate}[label=\textnormal{(\roman*)}]
\item For every coherent ideal \(\mathcal I\subset\mathcal O_X\), there is an integer \(a\ge0\), depending only on \((X,\Omega,\mathcal I)\), such that, for every smooth projective curve \(B/\kappa\) and morphism \(u:B\to X\) whose generic point avoids \(V(\mathcal I)\),
\[
    \deg u^{-1}V(\mathcal I)\le a\deg u^*\Omega.
\]

\item If \(\varphi:X\dashrightarrow\mathbb P^R\) is a fixed rational map defined on an open set \(U\), there is \(b\ge0\), depending only on the fixed map and \(\Omega\), such that
\[
\deg(\varphi\circ u)^*\mathcal O(1) \le b\deg u^*\Omega
\]
for every smooth projective curve \(B/\kappa\) and morphism \(u:B\to X\) whose generic point belongs to \(U\).
\end{enumerate}
\end{lemma}
\begin{proof}
By Serre's global-generation theorem
\cite[Chap~II, Thm~5.17]{hartshorne2013algebraic}, we choose \(a\) so that \(\mathcal I\otimes\Omega^{\otimes a}\) is globally generated.  The image of \(u^*\mathcal I\to\mathcal O_B\) is \(\mathcal O_B(-D)\), where \(D=u^{-1}V(\mathcal I)\). Its twist by \(u^*\Omega^{\otimes a}\) is globally generated, so \(-\deg D+a\deg u^*\Omega\ge0\).  This proves (i).

Fix a projective embedding of \(X\) with very ample bundle \(\mathcal A\). By \cite[Lec 7]{JoeHarris1992Alg.Geom},  the rational homogeneous coordinates of \(\varphi\) can be represented by  \(\sigma_0,\ldots,\sigma_R\in H^0(X,\mathcal M)\) where $\cM= A^{\otimes m}$ for some $m\ge 1$.  Since the generic point of \(u(B)\) belongs to \(U\), the pulled-back sections are not all zero generically.  Let \(D'\) be the scheme-theoretic common divisor of these sections on $B$.  The morphism \(\varphi\circ u\) extends to the proper curve \(B\), and
\[
 (\varphi\circ u)^*\mathcal O(1) \simeq u^*\mathcal M\otimes\mathcal O_B(-D').
\]
By the same theorem of Serre, there is \(b\) so that \(\Omega^{\otimes b}\otimes\mathcal M^{-1}\) is globally generated. It follows that 
\[
\deg  (\varphi\circ u)^*\mathcal O(1) \le \deg u^*\mathcal M\le b\deg u^*\Omega
\] 
proving (ii). 
\end{proof}

\begin{lemma}\label{lem:data-determinant}
The constant $C_a$ in \textnormal{(D3)} exists.
\end{lemma}

\begin{proof}
Put \(d_a=m_a^2\). Now let \(K_0=k_0(B_0)\) be a geometric function field with the normalizations of Section~2, and let \(x\in U_a(K_0)\).  Choose  any good basis \(\mathcal B_{n}\) by Lemma \ref{lem:good-basis-base-change}. We use the same ordered basis  for every finite extension \(E/K_0\) and at every completion of \(E\).

The algebraic determinant estimate of \cite[Thm~3.1]{looper2024arakelov}, with the conclusion as restated in \cite[Sec~3.1]{looper2026uniform}, is uniform over all good bases, the constant in its proof depends only on \(d_a=m_a^2\) and \(N_a\), not on the selected list or on \(n\). Therefore there is $C_a=C_a(d_a, N_a)$ such that for any $c(n)$-tuple $Q$ and $E/K_0$ where $E$ contains the field of definition of $Q$ there are $\eta_{n,w}\ge0$ for any place $w$ of $E$ such that
\[
g_{\mathcal B_{n},w}(Q)\ge r(F_a)_w-\eta_{n,w}, \qquad \eta_{n,w}\ge0,
\]
with
\[
\frac1{[E:K_0]}\sum_w\eta_{n,w} \le C_a\frac{\log n}{n}h_{\mathrm{coeff},K_0}(F_a).
\]

Let
\(\varphi_a:U_a\to\operatorname{Mor}_{m_a^2}(\mathbb P^{N_a},
\mathbb P^{N_a})\) be the coefficient morphism of the normalized lift. The  space $ \operatorname{Mor}_{m_a^2} (\mathbb P^{N_a}, \mathbb P^{N_a})$ is an open neighborhood in  $\P^t$ where $t= (N_a+1) \binom{N_a+m_a^2}{N_a}$. Embed this affine space naturally inside $\P^t$. Let \(X^*\) be the normal projective minimal compactification and let  \(X^*_{a,1}, \ldots,X^*_{a,s_a}\) be precisely the irreducible components with nonempty intersection with \( U_a\).  For each \(\alpha\), the nonempty open set \(U_a\cap X^*_{a,\alpha}\) is dense in the integral projective variety \(X^*_{a,\alpha}\).  Hence,  \(\varphi_a\) defines a fixed rational map
\[
\varphi_{a,\alpha}:X^*_{a,\alpha}\dashrightarrow\mathbb P^t.
\]
Let \(e_g>0\) be  an integer such that \(\Omega= \omega^{\otimes e_g}\) restricts to an ample line bundle on each component.   Since $X^*$ is proper the morphism $x:\Spec K_0 \to U_a$ extends to a morphism \(u:B_0\to X^*\). Since the generic point of $B_0$ maps to \(U_a\), the morphism $u$ factors through the unique component \(X^*_{a,\alpha}\) containing that generic image. Note that the generic image lies in the domain of \(\varphi_{a,\alpha}\). Lemma~\ref{lem:fixed-model-height} gives
\[
h_{\mathrm{coeff},K_0}(F_a) =\deg(\varphi_{a,\alpha}\circ u)^*\mathcal O(1) \le C'_{a,\alpha}\deg u^*\Omega =C'_{a,\alpha}e_g h_{\omega,K_0}(A).
\]
Put \(C'_a=\max_\alpha C'_{a,\alpha}\). Absorb \(C'_ae_g\) into \(C_a\) to obtain \eqref{eq:determinant-error}. See \cite[Prop~2.6, 2.8, 2.9]{looper2026uniform}.  

Finally, \(r(F_a)_w\) is a fixed rational multiple of \(-\log|\operatorname{Res}(F_a)|_w\).  The resultant is nonzero in the field of definition, so the product formula gives \eqref{eq:resultant-product}.
\end{proof}

\begin{lemma}\label{lem:data-boundary}
The model boundary functions satisfy \textnormal{(D4)}.
\end{lemma}
\begin{proof}
At a semistable place, the moduli morphism extends to the toroidal compactification by the semiabelian extension theorem \cite[Chap~IV]{faltings2013degeneration}.  If the reduction is bad, its closed point lies on the boundary.  On an affine chart with integral boundary generators \(f_1,\ldots,f_s\), each generator lies in the maximal ideal and at least one realizes the boundary-ideal valuation.  Hence
\[
  \lambda_{\partial X,v}(A)
  =d_v\min_i\operatorname{ord}_v(f_i(A))\ge d_v.
\]
\end{proof}

\subsection{Uniform local positivity by ultralimits}\label{sec:uniform-local}

\begin{definition}\label{def:admissible-system}
For each chart \(U_a\), let \(\mathcal Z_a\) be a class of pairs
\((A/C,Z)\), where \(C\) is algebraically closed and complete and contains
\(\kappa_p\) trivially, \(A/C\) is a \(1\)-centered bad-reduction fiber of
\(U_a\), and \(\varnothing\ne Z\subsetneq S(A)\) is closed.  The tuple
\(\mathcal Z=(\mathcal Z_1,\ldots,\mathcal Z_t)\) is \emph{admissible} if,
for every sequence in every \(\mathcal Z_a\), there is a nonprincipal
ultrafilter for which the sequence is compatible in the sense of
Definition~\ref{def:compatibility}.
\end{definition}

Choose a strictly increasing sequence $2\leq k_1<k_2<\cdots$ and set $h(n)=c(k_n)$.  We say a centered local pair $(A/C,Z)$ in a fixed chart is \emph{$n$-good} if
\begin{equation}\label{eq:n-good}
 m_{k_n}(A;Z)
 \geq r(F)+\frac1n\lambda_{\partial X}(A),
\end{equation}
where $m_{k_n}$ is defined in Subsection \ref{sec:local-energy}

\begin{theorem}\label{thm:uniform-local}
Let $\cZ_a$ be admissible.  For every increasing sequence $(k_n)$ there is a nonempty finite set $I_a\subset\Z_{\geq1}$ such that every pair in $\cZ_a$ is $n$-good for at
least one $n\in I_a$.
\end{theorem}

\begin{proof}
Suppose this is not true.  For each $j$, choose a pair $(A_j/C_j,Z_j)\in\cZ_a$ that is not $n$-good for $1\leq n\leq j$.  Put $\varepsilon_j= \lambda_{\partial X}(A_j)^{-1}$. By admissibility, choose a nonprincipal ultrafilter $\omega$ and a proper closed set $Z_\omega\subsetneq S(A_\omega)$.  By Lemma \ref{lem:boundary-degeneration}, $A_\omega$ has bad reduction over $H(\omega)$ and \(\lambda_{\partial X}(A_\omega)=1\). Theorem~\ref{thm:local-gap} gives
\[
 \gamma:=\gamma_{F_\omega}(A_\omega,Z_\omega)>0.
\]
Hence, there is $q_0$ such that
\[
 m_q(A_\omega;Z_\omega)-r(F_\omega)>\frac{2\gamma}{3} \qquad(q\geq q_0).
\]
Choose $n$ so large that $k_n\geq q_0$ and $1/n<\gamma/3$.  For every $j\geq n$, we have that $A_j/C_j$ is not $n$-good, thus
\[
 m_{k_n}(A_j;Z_j)-r(F_j)<\frac{\lambda_{\partial X}(A_j)}{n}.
\]
By the definition of the infimum $m_{k_n}(A;Z)$, we may take a tuple $Q_j$, whose entries retract into $Z_j$, such that
\[
 g_{\mathcal B_{k_n,j}}(Q_j)
 \leq m_{k_n}(A_j;Z_j)+\frac{\lambda_{\partial X}(A_j)}{j}.
\]
Thus,
\begin{equation}\label{eq:explicit-ultralimit-error}
 \varepsilon_j\bigl( g_{\mathcal B_{k_n,j}}(Q_j)-r(F_j)\bigr) <\frac1n+\frac1j.
\end{equation}
Lemma~\ref{lem:green-continuity} gives
\[
\lim_\omega\varepsilon_jg_{\mathcal B_{k_n,j}}(Q_j) =g_{\mathcal B_{k_n,\omega}}(Q_\omega), \qquad \lim_\omega\varepsilon_jr(F_j)=r(F_\omega).
\]
Because $\omega$ is nonprincipal, $\lim_\omega1/j=0$. Since the restriction
system is $\omega$-compatible, all entries of $Q_\omega$ retract into
$Z_\omega$.  Taking the ultralimit in
\eqref{eq:explicit-ultralimit-error} implies
\begin{align*}
 m_{k_n}(A_\omega;Z_\omega)-r(F_\omega) &\leq g_{\mathcal B_{k_n,\omega}}(Q_\omega)-r(F_\omega)\\
 &\leq\frac1n<\frac\gamma3.
\end{align*}
This contradicts the lower bound $2\gamma/3$ because $k_n\geq q_0$.
\end{proof}

\section{Compatibility of closed skeletal sets}\label{sec:compatibility}

A sequence of proper closed subsets may fill more and more of the skeleton, and its ultralimit may be the whole limiting skeleton. Here we now give some criterion on when a sequence of closed subsets is admissible.

Fix a nonprincipal ultrafilter \(\omega\).  Consider semistable abelian
varieties \(A_j/C_j\) over algebraically closed complete non-Archimedean
fields, together with proper closed sets \(Z_j\subsetneq S(A_j)\).  Assume
that the moduli sequence has the fundamental ultralimit of
Definition~\ref{def:ultralimit-ring}.

\begin{definition}\label{def:compatibility}
The collection $(A_j,Z_j)$ is \emph{$\omega$-compatible} if there is a proper
closed set $Z_\omega\subsetneq S(A_\omega)$ such that
\[
 r_{A_j}(x_j)\in Z_j\text{ for all }j
 \quad\Longrightarrow\quad
 r_{A_\omega}(x_\omega)\in Z_\omega
\]
for every sequence $(x_j)$.  A class of pairs is
\emph{admissible} if, for every sequence in the class, there exists a
nonprincipal ultrafilter for which it is compatible.
\end{definition}

\begin{proposition}\label{prop:separator}
Suppose there are continuous functions $\phi_j:S(A_j)\to[0,1]$ and \(\varphi_\omega:S(A_\omega)\to[0,1]\) such that:
\begin{enumerate}[label=(\alph*),leftmargin=2.3em]
\item $Z_j\subseteq\phi_j^{-1}(0)$,
\item  for every 
  sequence \(x_j\in A_j(C_j)\),
  \[
    \phi_\omega(r_{A_\omega}(x_\omega))
    =\lim_\omega\phi_j(r_{A_j}(x_j));
  \]
\item $\phi_\omega$ is not identically zero.
\end{enumerate}
Then $(A_j,Z_j)$ is $\omega$-compatible, with
$Z_\omega=\phi_\omega^{-1}(0)$.
\end{proposition}

\begin{proof}
If $r_{A_j}(x_j)\in Z_j$, then
$\phi_j(r_{A_j}(x_j))=0$.  Part (b) implies $\phi_\omega(r_{A_\omega}(x_\omega))=0$.  The zero locus is closed and is proper by (c).
\end{proof}

Now suppose the skeleta are identified with one compact metric torus \(T=\mathbb R^r/\mathbb Z^r\) by homeomorphisms
\( \theta_j:S(A_j)\xrightarrow\sim T\) and 
\(\theta_\omega:S(A_\omega)\xrightarrow\sim T.
\) Call the identifications \emph{ultralimit-compatible} if
\begin{equation}\label{eq:coordinate-compatibility}
\theta_\omega(r_{A_\omega}(x_\omega)) =\lim_\omega\theta_j(r_{A_j}(x_j))
\end{equation}
for every  sequence of points.  For closed $Y_j\subseteq T$, define their $\omega$-Hausdorff limit by
\begin{equation}\label{eq:hausdorff-limit}
 Y_\omega=\left\{\lim_\omega y_j:y_j\in Y_j\right\}.
\end{equation}
Since $T$ is compact  we get that $Y_\omega$ is closed.

\begin{theorem}\label{thm:uniform-hole}
Under \eqref{eq:coordinate-compatibility}, the restrictions $Z_j=\theta_j^{-1}(Y_j)$ are $\omega$-compatible if $Y_\omega\ne T$.  Equivalently, there are a point $t\in T$, a radius $\varepsilon>0$, and an $\omega$-large set of indices $J$ such that
\begin{equation}\label{eq:uniform-hole}
 Y_j\cap B(t,\varepsilon)=\varnothing\qquad(j\in J).
\end{equation}
\end{theorem}

\begin{proof}
Equation \eqref{eq:coordinate-compatibility} implies that for a sequence $(x_j)$ mapping to $Y_j$, its $\omega$-limit $x_\omega$ maps to $Y_\omega$. Assume $Y_\omega\ne T$ and choose $t\notin Y_\omega$.  As $Y_\omega$ is compact, $d(t,Y_\omega)>0$.  If no $\epsilon$ existed, then for every $n\geq1$ the set of $j$ for which $Y_j\cap B(t,1/n)\ne\varnothing$ would be $\omega$-large. Choose $y_j\in Y_j$ with $d(y_j,t)\to_\omega0$. Then $t\in Y_\omega$, a
contradiction. 
\end{proof}

\begin{corollary}\label{cor:fixed-subset}
Assume \eqref{eq:coordinate-compatibility}.  If there is one proper closed $Y\subsetneq T$ such that $\theta_j(Z_j)\subseteq Y$ for every $j$, then the collection is compatible, and one may take $Z_\omega=\theta_\omega^{-1}(Y)$.
\end{corollary}

\begin{proof}
Choose $t\notin Y$.  Since $Y$ is compact, a ball around $t$ misses $Y$ and hence misses every $\theta_j(Z_j)$.  Apply Theorem~\ref{thm:uniform-hole}.  
\end{proof}

\begin{example}\label{ex:filling-skeleton}
On \(T=\mathbb R/\mathbb Z\), let
\[
  Y_j=\{0,1/j,2/j,\ldots,(j-1)/j\}.
\]
Every \(Y_j\) is finite and proper, but its ultralimit is all of \(T\).  The associated restrictions are not compatible.  Thus pointwise properness alone is not sufficient.
\end{example}

\section{Height gap theorem}\label{sec:uniform-global-proof}

Fix the data of Theorem~\ref{thm:dynamical-data} and an admissible restriction
system $\cZ$. Put
\[
 M_{\mathrm{det}}=\max_a(C_a ),\qquad t=\#\{U_a\}.
\]
And let
\[
G_{g,p}=|\operatorname{GL}_{2g}(\mathbb Z/\ell_p\mathbb Z)|.
\]
The factor \(G_{g,p}\) bounds the degree of the field extension over which we have a trivial level structure.

\begin{lemma}\label{lem:positive-sequence}
For every $M_0>0$ there is a strictly increasing sequence $2\leq k_1<k_2<\cdots$ such that for all $a,n$, with \(\rho_{a,n}=\frac1{t\sqrt{k_n}}\), we get 
\begin{equation}\label{eq:rho-total}
 \sum_{a=1}^t\sum_{n\geq1}\rho_{a,n}<1
\end{equation}
and
\begin{equation}\label{eq:delta-an}
 \delta_{a,n}:=
 \frac{\rho_{a,n}}{G_{g,p}M_0n}
 -C_a \frac{\log k_n}{k_n}>0
\end{equation}
\end{lemma}
\begin{proof}
Put \(S_{3/2}=\sum_{n=1}^\infty n^{-3/2}<\infty\).  Choose \(D_0\ge1\) so
large that
\begin{equation}
  D_0^{-1/2}S_{3/2}<1,
  \qquad
  \frac{\sqrt {D_0}}{tG_{g,p}M_0}
  >M_{\mathrm{det}}\left(\log(2D_0)+\frac6e\right).
  \label{eq:D-choice}
\end{equation}
Such a \(D_0\) exists because \(\sqrt {D_0}/\log D_0\to\infty\).  Set
\(k_n=\lceil D_0n^3\rceil\).  This is strictly increasing, \(k_1\ge2\), and
\[
  \sum_{a,n}\rho_{a,n}
  =\sum_{n\ge1}\frac1{\sqrt{k_n}}
  \le D_0^{-1/2}S_{3/2}<1.
\]
Moreover, $D_0n^3\le k_n\le D_0n^3+1\le2D_0n^3$. The identity
\(\sup_{x\ge1}(\log x)/\sqrt x=2/e\) gives
\[
  \log k_n\le\log(2D_0)+3\log n
  \le\log(2D_0)+\frac6e\sqrt n.
\]
It follows from \eqref{eq:D-choice} that
\[
  \frac{\sqrt{k_n}}{tG_{g,p}M_0n}
  >M_{\mathrm{det}}\log k_n
  \ge C_a\log k_n.
\]
After division by \(k_n\), this is exactly \eqref{eq:delta-an}.
\end{proof}

Apply Theorem~\ref{thm:uniform-local} to this sequence and every chart, and
let $I_a$ be the resulting finite nonempty index set.  Set
\begin{equation}\label{eq:q-uniform}
 q:=\max\{k_n:a\in\{1,\ldots,t\},\ n\in I_a\},
 \qquad
 \delta:=\min_{a,n\in I_a}\delta_{a,n}>0.
\end{equation}

Let \(A/K\) be an abelian variety with no level structure.  Since \(\ell_p\) is prime to
the characteristic, \(A[\ell_p]\) is finite \'etale.  Its Galois
representation shows that there is a finite separable extension \(K'/K\)
over which \(A[\ell_p]\) is trivial and
\begin{equation}
  D=[K':K]\le G_{g,p}.
  \label{eq:level-degree}
\end{equation}
The Weil pairing supplies a full symplectic level structure over $K'$.  For \(w\mid v\), define
the \emph{boundary mass}
\begin{equation}
  \nu(w/v)=f(w/v)d_v=[\kappa(w):k].
  \label{eq:boundary-mass}
\end{equation}

\begin{definition}
The prescribed system $(Z_v)_{v\in\Sigma}$ is \emph{level-stably
admissible} if there exists a finite  separable extension trivializing
$A[\ell_p]$ such that, for every $v\in \Sigma$ and every place $w\mid v$, the  local pair belongs to
one of the classes $\cZ_a$.
\end{definition}

\begin{theorem}\label{thm:global-alternative}
Let \((A,\lambda)/K\) be a principally polarized abelian variety of dimension \(g\) with \(h_{\omega,K}(A)>0\).  Let \(\Sigma\) be a finite set of places of semistable reduction, let \((Z_v)_{v\in \Sigma}\) be a level-stably admissible system of proper closed sets, and assume
\[
\deg_k\Sigma\ge\frac{h_{\omega,K}(A)}{M_0}.
\]
For every universally \(\Sigma\)-restricted finite set \(P\subset A(\overline K)\) of distinct points with \(|P|\ge h^0(A,L^{\otimes q})\), either a nonzero section of \(L_{\overline K}^{\otimes q}\) vanishes on \(P\), or
\[
  \sum_{x\in P}\widehat h_{A,L,K}(x)
  \ge\delta h_{\omega,K}(A)>0.
\]
The constants \(q,\delta\) depend only on \(g,p,M_0\), and the restriction system, not on \(k,B,A\), or the level field.
\end{theorem}

\begin{proof}
Choose a separable extension \(K'/K\) of degree \(D\) satisfying \eqref{eq:level-degree} over which \(A\) has full level structure.  Let \(W\) be the set of places of \(K'\) above \(\Sigma\).  By definition, every local pair $(A_w, Z_w)$ belongs to some \(\mathcal Z_a\).  Theorem~\ref{thm:uniform-local} assigns to it an
\(n\in I_a\) for which it is \(n\)-good. Let \(W_{a,n}\) be the resulting partition, and set
\[
  M_{a,n}
  =\sum_{\substack{v\in\Sigma,\ w\mid v\\w\in W_{a,n}}}\nu(w/v).
\]
Since $D\le G_{g,p}$, we have,
\begin{equation}  \label{eq:mass-lower-bound}
 \sum_{a,n}M_{a,n}=  \sum_{v\in \Sigma, \,\,w\mid v} d_v f(w/v)\ge \frac D{G_{g,p}}\sum_v d_v  =\frac D{G_{g,p}}\deg_k\Sigma.
\end{equation}
Together with \eqref{eq:rho-total}, this implies that some pair \((a,n)\) satisfies
\begin{equation}  \label{eq:mass-pigeonhole}
M_{a,n} \ge\rho_{a,n}\frac D{G_{g,p}}\deg_k\Sigma.
\end{equation}

Assume that no nonzero section of \(L_{\overline K}^{\otimes q}\) vanishes on \(P\).
Since \(q\ge k_n\), no nonzero section of
\(L_{\overline K}^{\otimes k_n}\) vanishes on \(P\).   Lemma~\ref{lem:evaluation-selection} therefore gives a subset \(Q\subset P\) of cardinality \(c(k_n)\) with nonzero evaluation determinant. Choose a finite extension $E/K'$ over which all points of $Q$ are defined. Fix \(w\in W_{a,n}\) above \(v\in\Sigma\), and let \(u\) be a place of
\(E\) above \(w\).  The \(n\)-good inequality and local-degree scaling
give
\begin{equation}
  g_{\mathcal B_{k_n,a},u}(Q)
  \ge r(F_a)_u+\frac{e(u/w)f(u/w)}n
       \lambda_{\partial X,w}(A_{K'}).
  \label{eq:local-positive-before-mass}
\end{equation}
The abelian variety \(A_{K'}\) has semistable bad reduction at \(w\), and
its fine-level moduli map meets the toroidal boundary there.  Thus
\eqref{eq:boundary-order} applied directly over \(K'\) gives
\[
  \lambda_{\partial X,w}(A_{K'})
  \ge[\kappa(w):k]=f(w/v)d_v=\nu(w/v).
\]
Thus
\begin{equation}
  g_{\mathcal B_{k_n,a},u}(Q)
  \ge r(F_a)_u+
       \frac{e(u/w)f(u/w)\nu(w/v)}n.
  \label{eq:local-positive}
\end{equation}
At every place \(u\) not lying above \(W_{a,n}\), Theorem \ref{thm:dynamical-data} implies there are $\eta_{k_n,u}\ge 0$ such that
\[
  g_{\mathcal B_{k_n,a},u}(Q)
  \ge r(F_a)_u-\eta_{k_n,u}.
\]
For selected \(w\), we have $\sum_{u\mid w}e(u/w)f(u/w)=[E:K']$. Summing  over all places of \(E\) and
using the product formula gives us
\[
\sum_{u\in|B_E|}g_{\mathcal B_{k_n,a},u}(Q)  \ge\frac{[E:K']M_{a,n}}n -\sum_{u\in|B_E|}\eta_{k_n,u}.
\]
We divide by \(D[E:K']\), and use inequalities \eqref{eq:determinant-error} 
and  \eqref{eq:mass-pigeonhole} to get
\begin{align*}
  \frac1{D[E:K']}\sum_ug_{\mathcal B_{k_n,a},u}(Q)
  &\ge\frac{M_{a,n}}{Dn}
      -\frac{C_a\log k_n}{Dk_n}h_{\omega,K'}(A)\\
  &\ge\frac{\rho_{a,n}}{G_{g,p}n}\deg_k\Sigma
      -C_a\frac{\log k_n}{k_n}h_{\omega,K}(A)\\
  &\ge\left(
       \frac{\rho_{a,n}}{G_{g,p}M_0n}
       -C_a\frac{\log k_n}{k_n}
      \right)h_{\omega,K}(A)\\
  &=\delta_{a,n}h_{\omega,K}(A)
  \ge\delta h_{\omega,K}(A).
\end{align*}
Therefore
\[
  \sum_{x\in Q}\widehat h_{A,L,K}(x) = \frac{c(k_n) }{D[E:K']}\sum_ug_{\mathcal B_{k_n,a},u}(Q)
  \ge c(k_n)\delta h_{\omega,K}(A)
  \ge\delta h_{\omega,K}(A).
\]
N\'{e}ron--Tate heights are nonnegative, so the same lower bound holds after
enlarging \(Q\) to \(P\).
\end{proof}

\begin{remark}
In characteristic zero, a finite Galois-stable set of geometric points is a finite \'etale $K$-subscheme.  In characteristic $p$, pure inseparability is subtle.  For example, over $K=\F_p(t)$ the point $t^{1/p}\in\mathbf A^1(\overline K)$ is fixed by $\Gal(\overline K/K)$, and the linear form $X-t^{1/p}Y$ vanishes there, but no nonzero $K$-linear form does.  In other words, the fixed field of the absolute Galois group inside $\overline K$ is the perfect closure $K^{\perf}$, not $K$.
\end{remark}

\begin{proof}[Proof of Theorem~\ref{thm:descent}]
Since $P\subset X(K^\sep)$, we can choose a finite Galois extension $E/K$ containing all points of $P$.  The
subspace
\[
 W_E=\{s\in H^0(X_E,L_E):s(x)=0\text{ for all }x\in P\}
\]
is stable under the semilinear action of $\Gal(E/K)$.  Extension of scalars to $\overline K$ identifies $W_E\otimes_E \ol K$ with the corresponding geometric evaluation kernel, so $W_E\ne0$.  Faithfully flat Galois descent for vector
spaces gives a $K$-subspace $W_K\subset H^0(X,L)$ with
$W_K\otimes_KE=W_E$.  Hence $W_K\ne0$, and any nonzero element has the
required property.  
\end{proof}

\section{More on elliptic curves}\label{sec reductions}\label{sec joint}
Let $E$ be a semistable elliptic curve over $K$. Denote by $\Delta_E$  the minimal discriminant of $E$. For every closed point \(v\in|B|\), put
\[
n_{E,v}:=\operatorname{ord}_v(\Delta_E),
\qquad
T(E):=\{v\in|B|:n_{E,v}>0\},
\]
and write \(d_v=[\kappa(v):k]\). Since \(E\) is semistable, every \(v\in T(E)\) is a place of multiplicative reduction of Kodaira type \(I_{n_{E,v}}\). The minimal discriminant is a section of \(\omega_E^{\otimes12}\), and hence
\begin{equation}\label{eq:disc-hodge}
\deg_k\Delta_E =\sum_{v\in T(E)}d_vn_{E,v}= 12\deg_B(\omega_E) =12h_\omega(E).
\end{equation}
See \cite[Lem.~3.1]{GP22}.

\begin{theorem}\label{thm:pesenti-szapiro}\cite[Thm.~0.1]{pesenti2000inegalite}
Let $E$ be a nonisotrivial elliptic curve over $K$ with minimal discriminant  $\Delta_{E}$ and let $\mathcal{F}_{E/B}$ be its conductor. Then
\begin{equation}\label{eq:PS}
  \deg\Delta_E\leq 6p^e \bigl(2g(B)-2+\deg \mathcal{F}_{E/B}\bigr),
\end{equation}
where $p^{e}$ is the inseparability degree of the morphism $j:B\to \P^1$
\end{theorem}

Put $g=g(B)$. Assume $E$ is a semistable nonisotrivial elliptic curve over $K$ with  separable $j$-map. Then $T(E) $ is nonempty and the local conductor exponent is \(1\) at every bad place \cite[Prop.~1.1(2)]{pesenti2000inegalite}. Hence \(\mathcal{F}_{E/B}=\sum_{v\in T(E)}[v]\) and \(\deg \mathcal{F}_{E/B} =\deg T(E)\). Then equation \eqref{eq:PS} gives the bound
\begin{equation}\label{eq:support-height}
 h_{\omega,K}(E) = \frac{ \deg\Delta_E}{12}\le g-1+\frac{\deg T(E)}{2} \le (g+1)\deg T(E).
\end{equation}

\begin{proposition}\label{prop:smallplaces}
Let $E,F$ over $K$ be nonisotrivial semistable elliptic curves with separable $j$-maps. Let $M\geq48(g+1)$. Let \(Q=\ell^\alpha>M\), where \(\ell\ne p\), and suppose $E[Q]\simeq F[Q]$ as $G_K$-modules.
Then the set
\[
 \Sigma_{\mathrm{sm}}= \{v:1\leq n_{E,v},n_{F,v}\leq M\}
\]
satisfies
\begin{equation}\label{eq:small-mass}
 \deg_k\Sigma_{\mathrm{sm}}\geq \frac{h_{\omega,K}(E)+h_{\omega,K}(F)}{4(g+1)}.
\end{equation}
\end{proposition}

\begin{proof}
Put $H=h_{\omega,K}(E)+h_{\omega,K}(F)$. By \eqref{eq:support-height},
\[
 \deg_k(T(E)\cup T(F)) \geq \max\{\deg_k T(E),\deg_k T(F)\} \geq \frac{H}{2(g+1)}.
\]
Equation \eqref{eq:disc-hodge} shows that the places at which either $n_{E,v}>M$ or $n_{F,v}>M$ have total degree at most
\[
 \frac{12h_\omega(E)}M+\frac{12h_\omega(F)}M \leq\frac{H}{4(g+1)}.
\]
After removing them, a remaining place in the union is bad for at least one curve, with positive discriminant exponent strictly smaller than $Q$.

It remains to show that if $1\le n_{E,v}\le M$ and $n_{F,v}\le M$, then $1\le n_{F,v}$ and vice versa.  A nonsplit multiplicative torus splits over an unramified quadratic extension \cite[App C, Thm 14.1]{silverman2009arithmetic}, so passing to the split form does not change inertia.  At a multiplicative place with Tate exponent $n$, by Lemma \ref{lem:tate-monodromy} tame inertia on prime-to-$p$ $Q$-torsion contains the matrix 
\[
 \begin{pmatrix}1&n\\0&1\end{pmatrix}\pmod Q;
\]
this is nontrivial when $0<n<Q$.  At a good place inertia is trivial by the N\'eron--Ogg--Shafarevich criterion \cite[Thm~1]{serre1971good}.  Isomorphic $G_K$-modules cannot have these two different inertia actions. Thus every remaining place is jointly multiplicative.
\end{proof}

\subsection{Reductions}

\begin{proposition}
\label{prop:perfect-constant}
Assume that  Theorem \ref{thm:main} is known with the constant $C(b,p)$, whenever the constant field $k$ is perfect. Then it holds over every field $k$ of characteristic exponent $p$.
\end{proposition}

\begin{proof}
There is nothing to prove when $p=1$, so assume $p>1$. Let
\[
k^{\perf}=\bigcup_{r\geq 0}k^{1/p^r}
\]
be the perfect closure of $k$ in a fixed algebraic closure. Put $B'=B\times_k k^{\perf}$, and  $K'=Kk^{\perf}=k^{\perf}(B')$. Then $B'$ is a smooth projective geometrically integral curve over $k^\perf$. Proper flat cohomology and base change gives
\[
 H^1(B',\mathcal O_{B'})\simeq H^1(B,\mathcal O_B)\otimes_k k^{\perf},
\]
proving that $g(B')=g(B)$. The extension \(K'/K\) is purely inseparable. Hence \(\operatorname{Spec} K'\to\operatorname{Spec}K\) is a universal homeomorphism. By topological invariance of the \'etale site \cite[Exp~VIII, Thm~1.1]{artin1973theorie}, pulling back induces an equivalence between the categories of finite etale \(K\)-schemes and finite etale \(K'\)-schemes. After choosing compatible separable closures, this identifies \(G_{K'}\) with \(G_K\). In particular, for every integer \(N\) prime to \(p\), it identifies the \(N\)-torsion modules of $E[N]$ and $E_{K'}[N]$. The pullback of the nonconstant \(j\)-map remains nonconstant. For every closed point $b'\in B'$ lying over $b\in B$, the local homomorphism
\[
\mathcal O_{B,b}\longrightarrow \mathcal O_{B',b'}
\]
is a faithfully flat extension of discrete valuation rings. Consequently, semistable reduction is preserved by
\cite[\S 7.4 Cor.~4]{BLR}. 

Apply the perfect-field theorem to $E_{K'}$ and $F_{K'}$. It produces a $K'$-isogeny $E_{K'}\to F_{K'}$. The extension $K'/K$ is primary, meaning that the algebraic closure of $K$ in $K'$ is purely inseparable over $K$. For abelian varieties $A_1,A_2/K$, base change across a primary extension induces a bijection
\[
\Hom_K(A_1,A_2)\xrightarrow{\sim}\Hom_{K'}((A_1)_{K'},(A_2)_{K'})
\]
by \cite[Lem~1.2.1.2]{chaiconoort}. The $K'$-isogeny therefore descends to a homomorphism $f:E\to F$ over $K$. Since $f_{K'}$ is nonzero, we get that \(f\ne0\),
and every nonzero homomorphism between elliptic curves is an isogeny.
\end{proof}

\begin{lemma}
\label{lem:primary-decomp}
Fix a number $C_0\geq 2$. Suppose that, for every prime $\ell\neq p$, every integer $\alpha\geq 1$, and every pair of nonisogenous semistable elliptic curves $E,F$ over $K$, an isomorphism \(E[\ell^\alpha]\simeq F[\ell^\alpha]\) implies $\ell^\alpha\leq C_0$. Define
\[
 C_1(C_0,p):=
\prod_{\substack{\ell\leq C_0\\ \ell\ \mathrm{prime},\ \ell\neq p}} \ell^{\lfloor\log_\ell C_0\rfloor}.
\]
Then, for every integer \(N\) prime to \(p\) and every pair of nonisogenous semistable elliptic curves \(E,F\), an isomorphism \(E[N]\simeq F[N]\) implies
\[
N\leq C_1(C_0,p).
\]
\end{lemma}

\begin{proof}
Write $N=\prod_{\ell\mid N}\ell^{\alpha_\ell}$. A $(\mathbb Z/N\mathbb Z)$-linear isomorphism $E[N]\to F[N]$ restricts to an isomorphism \(E[\ell^{\alpha_\ell}]\simeq F[\ell^{\alpha_\ell}]\) for every $\ell\mid N$. By hypothesis, $\ell^{\alpha_\ell}\leq C_0$. Hence $\ell\leq C_0$ and $\alpha_\ell\leq\lfloor\log_\ell C_0\rfloor$. Multiplying these numbers implies $N\le C_1(C_0,p)$.
\end{proof}

Assume for the rest of this subsection that $k$ is perfect of characteristic $p>0$. For an elliptic curve $C/K$, write $C^{(p)}$ for the pullback of $C$ by the absolute Frobenius of $\Spec K$. The relative Frobenius and its dual, the Verschiebung \cite[Chap 12]{katzmazur1985},
are denoted
\[
F_{C/K}:C\longrightarrow C^{(p)},
\qquad V_{C/K}:C^{(p)}\longrightarrow C.
\]

\begin{lemma}
\label{lem:frob-invariance}
The relative Frobenius is a finite flat purely inseparable isogeny of degree $p$, and
\[
V_{C/K}\circ F_{C/K}=[p]_C,
\qquad F_{C/K}\circ V_{C/K}=[p]_{C^{(p)}}.
\]
If $(m,p)=1$, then $F_{C/K}$ induces a $G_K$-equivariant isomorphism
\[
C[m]\xrightarrow{\sim}C^{(p)}[m].
\]
\end{lemma}

\begin{proof}
The assertions about relative Frobenius follow from \cite[Chap~II, Prop~2.11]{silverman2009arithmetic} and the dual-isogeny identities follow from \cite[Chap~III, Thm~6.1, Thm~6.2]{silverman2009arithmetic}. Let $a\in\mathbb Z$ be such that $ap\equiv 1\pmod m$. On $m$-torsion,
\[
([a]V_{C/K})\circ F_{C/K}=[ap]=[1], \qquad F_{C/K}\circ([a]V_{C/K})=[ap]=[1].
\]
Thus $[a]V_{C/K}$ is a two-sided inverse. Both morphisms are defined over $K$, so the isomorphism is $G_K$-equivariant.
\end{proof}

Since $k$ is perfect and $K/k$ has transcendence degree one, $j_E$ is generically separable if and only if $dj_E\neq 0$ in $\Omega^1_{K/k}$, and this is equivalent to $j_E\notin K^p$. 

\begin{proposition}
\label{prop:frob-minimal}
Let $E/K$ be a nonisotrivial semistable elliptic curve. There are a semistable elliptic curve $E_0/K$ where $j_{E_0}$ is generically separable, an integer $r\geq 0$, and a purely inseparable $K$-isogeny
\[
E_0\longrightarrow E
\]
of degree $p^r$. Consequently, $E_0[m]\simeq E[m]$ as $G_K$-modules for every $m$ prime to $p$.
\end{proposition}

\begin{proof}
Since \(j(E)\) is nonconstant, its divisor is nonzero. If \(j(E)\in K^{p^n}\), then every coefficient of the divisor of \(j(E)\) is divisible by \(p^n\). Consequently, there are only finitely many such \(n\). Let \(r\geq0\) be maximal such that
\[
j(E)\in K^{p^r}.
\]
Then \(j(E)\notin K^{p^{r+1}}\). It remains to construct one predecessor whenever $j(E)=j_0^p$ with $j_0\in K$. By \cite[Lec~1, \S 2, Lem~2.1]{ulmer2011}, there is an elliptic curve \(C/K\) and a \(K\)-isomorphism
\[
\iota:C^{(p)}\xrightarrow{\sim}E.
\]
Since semi-abelian reduction is invariant under isogeny \cite[\S 7.3, Cor.~7]{BLR}, the curve \(C\) is semistable. Repeat the construction exactly $r$ times. The resulting curve $E_0$ has $j(E_0)\notin K^p$, hence the $j$-invariant of $E_0$ is generically separable. The composite isogeny has degree $p^r$, and the assertion about $m$-torsion follows by Lemma~\ref{lem:frob-invariance}.
\end{proof}

\section{Skeletal restrictions and admissibility}\label{sec skeletal}\label{sec admissibility}

\subsection{The Weil pairing}
Let \(K\) be a field, let \(E/K\) be an elliptic curve, and let \(N\geq2\) be invertible in \(K\). We write \(E[N]\) for \(E[N](\overline K)\). The Weil pairing is a bilinear, alternating, perfect, and \(G_K\)-equivariant pairing
\[
e_{N,E}:E[N]\times E[N]\longrightarrow\mu_N,
\]
where $\mu_N$ is the set of $N$-th roots of unity in $\ol K.$ In particular, it is surjective. See \cite[Chap.~III, Prop.~8.1]{silverman2009arithmetic}.

\begin{proposition}\label{prop:weilparameterexists}
Let \(E,F/K\) be elliptic curves, let \(N\geq2\) be invertible in \(K\), and let
\[
\phi:E[N]\xrightarrow{\sim}F[N]
\]
be a \(\mathbb Z/N\mathbb Z\)-linear isomorphism. There is a unique unit \(r_\phi\in(\mathbb Z/N\mathbb Z)^\times\) such that
\[
e_{N,F}\bigl(\phi(P_1),\phi(P_2)\bigr) =e_{N,E}(P_1,P_2)^{r_\phi}
\]
for all \(P_1,P_2\in E[N]\).
\end{proposition}
\begin{proof}
Choose a basis \((P,Q)\) of the free \(\mathbb Z/N\mathbb Z\)-module \(E[N]\), and put \(\zeta=e_{N,E}(P,Q)\). Since the Weil pairing is perfect, we get that \(\zeta\) is a primitive \(N\)-th root of unity. Since \((\phi(P),\phi(Q))\) is a basis of \(F[N]\), there is a unique \(r_\phi\in\mathbb Z/N\mathbb Z\) such that
\[
e_{N,F}\bigl(\phi(P),\phi(Q)\bigr)=\zeta^{r_\phi}.
\]
The pairing on \(F[N]\) is perfect, so \(\zeta^{r_\phi}\) is also primitive and \(r_\phi\) is a unit inside $(\Z/N\Z)^\times$. If \(P_1=xP+yQ\) and \(P_2=zP+wQ\), then
\[
e_{N,E}(P_1,P_2)=\zeta^{xw-yz}
\]
and
\[
e_{N,F}\bigl(\phi(P_1),\phi(P_2)\bigr) =\zeta^{r_\phi(xw-yz)}.
\]
This completes the proof.
\end{proof}

\begin{lemma}\label{lem:putweiparameterinasmallset}
Let $E,F/K$ be elliptic curves. Let $N\ge2$, and let $\phi:E[N]\xrightarrow{\sim}F[N]$ be a $\Z/N\Z$-linear isomorphism. For $a\in(\Z/N\Z)^\times$,
\[
 r_{[a]\circ\phi}=a^2r_\phi.
\]
If $N=\ell^\alpha$, scalar multiplication can therefore put $r_\phi$ in the following list of square-class representatives:
\[
 \begin{cases}
 \{1,\varepsilon_\ell\},& \ell \text{ odd},\text{ where }\left(\frac {\varepsilon_\ell}{ \ell} \right) =-1 \\
 \{1\},&\ell=2,\ \alpha=1,\\
 \{1,-1\},&\ell=2,\ \alpha=2,\\
 \{1,-1,3,-3\},&\ell=2,\ \alpha\ge3,
 \end{cases}
\]
\end{lemma}
\begin{proof}
Since $a\in (\Z/N\Z)^\times$, multiplication by $a$ is an automorphism of the $N$-torsions. By bilinearity,
\[
e_{N,F}\bigl([a]\varphi(P_1),[a]\varphi(P_2)\bigr) =e_{N,F}\bigl(\varphi(P_1),\varphi(P_2)\bigr)^{a^2} =e_{N,E}(P_1,P_2)^{a^2r_\varphi}.
\]
Proposition~\ref{prop:weilparameterexists} therefore gives \(r_{[a] \circ\varphi}= a^2r_\varphi\). For odd \(\ell\), a unit modulo \(\ell^\alpha\) is a square if and only if its reduction modulo \(\ell\) is a square. Thus there are exactly two square classes. For \(\ell=2\), the cases \(\alpha=1,2\) are immediate. If \(\alpha\geq3\), an odd unit is a square modulo \(2^\alpha\) if and only if it is congruent to \(1\pmod8\). Thus square classes are determined by reduction modulo \(8\), giving representatives \(1,-1,3,-3\). See \cite[Chap 11.4]{lozano2019number}.
\end{proof}

For any isomorphism $\phi$, we replace it by the isomorphism $\psi:=[a]\circ \phi$ where $r_\psi$ is in a fixed set given by Lemma \ref{lem:putweiparameterinasmallset}. Once this global value of $r_\psi$ is fixed over the function field $K$, for any valuation $v$ we have the equality of local Weil multipliers $r_{\psi,v} = r_\psi$.

\subsection{Tate monodromy}
Let \(L\) be the completion of the maximal unramified extension of \(K_v\). A multiplicative elliptic curve becomes a split Tate curve over an unramified extension of degree at most two \cite[App C, Thm 14.1]{silverman2009arithmetic}. Since the residue field of \(L\) is separably closed, the reduction is split over \(L\). A split multiplicative elliptic curve over $L$ is the Tate curve
\[
E(\overline L)\simeq \overline L^\times/q_E^{\mathbb Z}, \qquad n_E=\ord(q_E)>0.
\]
Put $\nu=-\log\abs{\cdot}$. If $\rho_v=\nu(\pi)$ for a uniformizer, then the metric analytic skeleton is $\mathbb R/(\rho_vn_E)\mathbb Z$, and the metric retraction is $[z]\mapsto\nu(z)$. We shall usually divide the metric coordinate by $\rho_v$ and use the equivalent skeleton $\mathbb R/n_E\mathbb Z$. In this rescaled coordinate, the retraction and its normalized version are
\[
r_E([z])=\ord(z)\pmod{n_E\mathbb Z}, \qquad \tau_E([z])=\frac{\ord(z)}{n_E}\pmod{\mathbb Z} =\frac{\nu(z)}{\nu(q_E)}\pmod{\mathbb Z}\in\mathbb R/\mathbb Z.
\]
Thus every normalized assertion below is independent of the valuation normalization. If the valuation is renormalized after a finite local extension, both \(\nu(z)\) and \(\nu(q_E)\) are multiplied by the same positive constant, so their ratio is unchanged. The Tate uniformization and its Galois compatibility are given in \cite[Chap~V, Thm~3.1, Rmk~3.1.2]{SilvermanAdvanced}.  For the general analytic skeleton and retraction map, see \cite[Sec~6.5]{berkovich1990spectral}.

\begin{lemma}
\label{lem:tate-monodromy}
Let $Q=\ell^\alpha$ where  $(p,\ell)=1$. Choose a primitive $Q$th root of unity $\zeta_Q$ and a $Q$th root $q_E^{1/Q}$. Relative to the basis
\(
P_E=[\zeta_Q]\) and \(R_E=[q_E^{1/Q}]
\)
of $E[Q]$, there is a tame inertia element $\sigma \in \Gal(\ol L/L)$ whose matrix is
\begin{equation}
T_E=\begin{pmatrix}1&n_E\\0&1\end{pmatrix}
\quad\text{in }\operatorname{GL}_2(\mathbb Z/Q\mathbb Z).
\label{eq:tate-monodromy}
\end{equation}
\end{lemma}

\begin{proof}
Choose a uniformizer $\pi$ of $L$. Since $Q$ is invertible in the residue field and that field is separably closed, Hensel's lemma shows that $\mu_Q\subset L$ and that the $Q$th-power map on $\mathcal O_L^\times$ is surjective. Hence $q_E=\pi^{n_E}\eta_E^Q$ for some $\eta_E\in\mathcal O_L^\times$. The polynomial $X^Q-\pi$ is Eisenstein, so $L(\pi^{1/Q})/L$ is totally ramified of degree $Q$. Since $\ell\ne p$, it is tamely ramified. Since $\mu_Q\subset L$, the extension is cyclic Galois, and
\[
\tau(\pi^{1/Q})=\zeta_Q\pi^{1/Q}
\]
defines an element of its Galois group. The restriction map from
\(G_L\) to this finite Galois group is surjective, so choose a lift
\(\sigma\) of \(\tau\). 
Take
\(
q_E^{1/Q}=(\pi^{1/Q})^{n_E}\eta_E.
\)
Then
\[
\sigma(q_E^{1/Q})=\zeta_Q^{n_E}q_E^{1/Q},
\qquad \sigma(\zeta_Q)=\zeta_Q.
\]
Thus \(\sigma(P_E)=P_E\) and
\(\sigma(R_E)=n_EP_E+R_E\), which gives (\ref{eq:tate-monodromy}).
\end{proof}

When \(E\) and \(F\) are considered simultaneously, choose the same \(\pi\), \(\pi^{1/Q}\), \(\zeta_Q\), and lift \(\sigma\), and write
\[
q_A=\pi^{n_A}\eta_A^Q,\qquad R_A=[(\pi^{1/Q})^{n_A}\eta_A], \qquad A\in\{E,F\}.\] We call the resulting Tate bases compatible. The same element \(\sigma\) then has matrices \(T_E\) and \(T_F\) in the two bases.
Let \( \phi:E[Q]\longrightarrow F[Q]\) be a local Galois-equivariant isomorphism. In compatible Tate bases, write
\[
[\phi]=\begin{pmatrix}a&b\\c&d\end{pmatrix}.
\]

\begin{lemma}
\label{lem:local-congruences}
With the notation above,
\begin{align}
n_Fc&\equiv0\pmod Q,\label{eq:local-9}\\
n_Ec&\equiv0\pmod Q,\label{eq:local-10}\\
an_E&\equiv dn_F\pmod Q.\label{eq:local-11}
\end{align}
Let $r_\phi\in(\mathbb Z/Q\mathbb Z)^\times$ be the Weil-pairing multiplier of $\phi$. Then
\begin{equation}
d^2n_F\equiv r_\phi n_E\pmod Q.
\label{eq:local-12}
\end{equation}
\end{lemma}

\begin{proof}
Galois equivariance for the inertia element in Lemma~\ref{lem:tate-monodromy} is the matrix identity 
\[
\begin{pmatrix}a&b\\c&d\end{pmatrix}
\begin{pmatrix}1&n_E\\0&1\end{pmatrix}
=
\begin{pmatrix}1&n_F\\0&1\end{pmatrix}
\begin{pmatrix}a&b\\c&d\end{pmatrix}.
\]
Hence
\[
\begin{pmatrix}a&an_E+b\\c&cn_E+d\end{pmatrix}=
\begin{pmatrix}a+n_Fc&b+n_Fd\\c&d\end{pmatrix}.
\]
Equality of their entries gives \eqref{eq:local-9}, \eqref{eq:local-10}, and \eqref{eq:local-11}. The compatible bases use the same root $\zeta_Q$ and the same convention for the Tate uniformization, so
\[
 e_{Q,E}(P_E,R_E)=e_{Q,F}(P_F,R_F)=\zeta_Q.
\]
The Weil pairing on prime-to-$p$ torsion is bilinear, alternating, perfect, and Galois equivariant by \cite[Chap~III, Prop~8.1]{silverman2009arithmetic}. For a rank-two free module with such a pairing, a linear isomorphism multiplies the pairing by its determinant.  Hence $r_\phi=\det(\phi)=ad-bc$. Therefore
\[
d^2n_F\equiv adn_E\equiv (ad-bc)n_E \equiv r_\phi n_E\pmod Q.
\]
This proves \eqref{eq:local-12}.
\end{proof}

\begin{lemma}
\label{lem:square-roots}
Let $m\geq1$, let $\ell$ be prime, and let $u\in(\mathbb Z/\ell^m\mathbb Z)^\times$. The congruence $x^2\equiv u\pmod{\ell^m}$ has at most two solutions if $\ell$ is odd and at most four solutions if $\ell=2$.
\end{lemma}

\begin{proof}
See \cite[Chap 11.4]{lozano2019number}.
\end{proof}

\begin{proposition}\label{prop:diagonal-retraction}
Use the notation of Proposition~\ref{prop:smallplaces}. Suppose \(Q=\ell^\alpha>M^2\). Set \(i_0=\lfloor\log_\ell M\rfloor\) and \( Q_0=\ell^{\alpha-i_0}\).  There are integers $1\leq u_E,u_F\leq M$, a unit $\delta_0\in(\mathbb Z/Q_0\mathbb Z)^\times$, and a set $\Sigma\subset\Sigma_{\sm}$ such that
\begin{equation}\label{eq:diagonal-mass}
\deg_k\Sigma\geq\frac{H}{16(g+1)M^3},
\end{equation} 
and, for every $v\in\Sigma$,
\[
n_{E,v}=u_E,\qquad n_{F,v}=u_F,
\]
while, relative to one compatible choice of Tate bases at each $v$, the matrix of $\phi|_{E[Q_0]}$ is upper triangular with lower-right entry $\delta_0$.
\end{proposition}

\begin{proof}
Fix $v\in\Sigma_{\sm}$ and write $\gamma_v$ and $\delta_v$ for the lower row of the local matrix of $\phi$ at $v$. Since $n_{E,v}\leq M$, we have $v_\ell(n_{E,v})\leq i_0$. Congruence~\eqref{eq:local-10} implies
\[
v_\ell(\gamma_v)\geq\alpha-v_\ell(n_{E,v})\geq\alpha-i_0.
\]
Thus $\gamma_v\equiv0\pmod{Q_0}$. Thus the local matrix is upper triangular. Since it remains invertible, both diagonal entries are units modulo $Q_0$.
The inequality $Q>M^2$ and the definition of $i_0$ imply
\[
Q_0=Q/\ell^{i_0}\ge Q/M>M.
\]
Equation \eqref{eq:local-12} implies $\delta_v^2n_{F,v}\equiv r_\phi n_{E,v}\pmod Q$. The integers $n_{E,v}$ and $n_{F,v}$ are nonzero and strictly smaller than $Q_0$, while $\delta_v$ and $r_\phi$ are units. If their $\ell$-adic valuations differed, the two sides of \eqref{eq:local-12} would have different valuations below $v_\ell(Q_0)$, which is impossible. Hence
\[
v_\ell(n_{E,v})=v_\ell(n_{F,v})=:s_v.
\]
After dividing the equation \eqref{eq:local-12} by $\ell^{s_v}$,  Lemma~\ref{lem:square-roots} gives at most four possible residues for $\delta_v$ modulo $\ell^{\alpha-i_0-s_v}$. Each such residue has $\ell^{s_v}\leq n_{E,v}\leq M$ lifts modulo $Q_0$. Therefore, for each of the at most $M^2$ pairs $(n_{E,v},n_{F,v})$, there are at most $4M$ possible lower-right entries $\delta_v$ modulo $Q_0$.

Partition $\Sigma_{\sm}$ by the labels $(n_{E,v},n_{F,v},\delta_v)$. There are at most $4M^3$ labels. So one label occurs on a subset $\Sigma$ satisfying
\[
\deg_k\Sigma\geq\frac{\deg_k\Sigma_{\sm}}{4M^3}
\geq\frac{H}{16(g+1)M^3},
\]
where the second inequality is Proposition~\ref{prop:smallplaces}. This completes the proof.
\end{proof}

\begin{proposition}\label{prop:a0 exists} 
Use the notation as in the previous Proposition.  Choose $a_0\delta_0\equiv1\pmod{Q_0}$ and put \(\psi=[a_0]\circ\phi|_{E[Q_0]}\). For every $v\in\Sigma$, every point of the graph
\[
\Gamma_\psi=\{(x,\psi(x)):x\in E[Q_0]\}\subset(E\times F)[Q_0]
\]
retracts to the diagonal
\begin{equation}
\Delta_{u_E,u_F}=\{(u_Et,u_Ft):t\in\mathbb R/\mathbb Z\} \subset(\mathbb R/u_E\mathbb Z)\times(\mathbb R/u_F\mathbb Z).
\label{eq:weighted-diagonal}
\end{equation}
\end{proposition}

\begin{proof}
At $v\in\Sigma$, let $(P_E,R_E)$ and $(P_F,R_F)$ be the $Q_0$-Tate bases. Every $x\in E[Q_0]$ has a unique expression
\[
x=jP_E+iR_E,\qquad i,j\in\mathbb Z/Q_0\mathbb Z.
\]
The normalized skeletal coordinate of $x$ is $i/Q_0$ modulo $\mathbb Z$. The local matrix of $\psi$ is upper triangular and its lower-right entry is $a_0\delta_0=1$. Hence the coefficient of $R_F$ in $\psi(x)$ is also $i$, and
\[
(\tau_E(x),\tau_F(\psi(x)))
=\left(\frac{i}{Q_0},\frac{i}{Q_0}\right)\pmod{\mathbb Z^2}.
\]
In the unnormalized skeleta this point is $(u_Ei/Q_0,u_Fi/Q_0)$, which lies on \eqref{eq:weighted-diagonal}. This proves the final assertion.
\end{proof}

The orientation of the skeleton of a multiplicative elliptic curve is not intrinsic. To obtain a restriction independent of the two choices of orientation, define
\begin{equation}\label{eq:signed-diagonal}
\Delta^{\pm}_{u_E,u_F}=\Delta_{u_E,u_F}\cup\{(u_Et,-u_Ft):t\in\mathbb R/\mathbb Z\}.
\end{equation}
Proposition~\ref{prop:a0 exists} shows that the graph lies on one of its two branches at every common multiplicative place for $E$ and $F$ with bounded exponent.

\begin{remark}
\label{rem:level-ratio}
The ratio between the original and lowered levels satisfies
\begin{equation}
1\leq Q/Q_0=\ell^{i_0}\leq M.
\label{eq:level-ratio}
\end{equation}
Thus a uniform bound for $Q_0$ gives a uniform bound for $Q$.
\end{remark}

\subsection{Admissibility}\label{sec:fm-admissibility}

Let \(E\) be an elliptic curve. Denote by
\(\lambda_E:E\to E^\vee\) its canonical principal polarization.
For elliptic curves \(E\) and \(F\), we equip \(E\times F\) with the product principal polarization
\[
\lambda:=\lambda_{E\times F}
=
\lambda_E\times\lambda_F,
\]
that is an isomorphism from $E\times F$ to $E^\vee\times F^\vee$. If \(L_E\) and \(L_F\) are line bundles representing
\(\lambda_E\) and \(\lambda_F\), respectively, then the line bundle $\cL_A$
representing \(\lambda_{E\times F}\) is
\[
p_E^*L_E\otimes p_F^*L_F,
\]
where $p_E$ and $p_F$ are natural projections of $E\times F $ to $E$ and $F$. We retain the line bundle $L_A=\mathcal M_\lambda^{\otimes3}$ from \eqref{eq:intro-line-bundle}. For a semistable bad-reduction place $v$, we work over $C_v=\widehat{\overline{K_v}}$. Let $S(A_{C_v})$ be the canonical Raynaud skeleton, and let
\[
 r_{A,v}:A(C_v)\longrightarrow S(A_{C_v})
\]
be its retraction map. 

We first prove the compatibility of Tate coordinates with the fundamental ultralimit. Let $C_j$ be algebraically closed complete non-Archimedean fields. Let $E_j$ and $F_j$ be elliptic curves over $C_j$ with multiplicative reductions over $C_j$ and Tate parameters $q_j,q_j'\in C_j$ respectively. Fix a nonprincipal ultrafilter $\omega$.

\begin{lemma}\label{lem:tatecurveultralimit}
Let \(\epsilon_j\) be a sequence where \(0< \lim_\omega \epsilon_j \nu_j(q_{j}) <\infty \). Then there exists $q_\omega\in H(\omega)$ such that it is the $(\epsilon_j)$-degeneration of $q_j$ and
\[
\nu(q_\omega )= \lim_\omega \epsilon_j \nu_j(q_j) 
\]
Moreover, the ultralimit of the Tate curves \(E_j\) is the Tate curve \(E_\omega\) over \(H(\omega)\) with parameter \(q_\omega\). In particular, \(E_\omega\) has split multiplicative reduction.
\end{lemma}

\begin{proof}
Since \(q_j\) is a Tate parameter, \(|q_j|_j<1\). Hence $\sup_j |q_j|_j^{\epsilon_j}\leq 1$, so \(q=(q_j)_j\) belongs to \(U^\epsilon\). Let \(q_\omega\) be its image in \(H(\omega)\). By the definition of the norm on \(H(\omega)\),
\[
|q_\omega|_\omega = \lim_\omega |q_j|_j^{\epsilon_j} = \exp\left(-\lim_\omega\epsilon_j\nu_j(q_j)\right).
\]
The hypothesis therefore gives $0<|q_\omega|_\omega<1$, and consequently
\[
\nu_\omega(q_\omega) = -\log |q_\omega|_\omega = \lim_\omega \epsilon_j\nu_j(q_j).
\]
Put $L=\lim_\omega\epsilon_j\nu_j(q_j)>0$. Choose $(0<r<1)$ such that \(|q_j|_j^{\epsilon_j}\leq r\) on an $\omega$-large set. If $f(q)=\sum_{n\geq 0}c_nq^n$ is an integral $q$-expansion, then $|c_n|_j\leq 1$, and the non-Archimedean triangle inequality gives
\[
\left|\sum_{n>N}c_nq_j^n\right|_j^{\epsilon_j} \leq r^{N+1}
\]
on this $\omega$-large set. Thus the tails uniformly tend to zero. Therefore $\left(f(q_j)\right)_\omega=f(q_\omega)$. Recall that the Tate curve with parameter $q$ is given by a Weierstrass equation whose coefficients are convergent integral $q$-expansions \cite[Chap. V, Thm. 3.1]{SilvermanAdvanced}. The Tate equation is $y^2+xy=x^3+a_4(q)x+a_6(q)$, where $ a_4,a_6\in\Z[[q]]$, and its discriminant is
\[
 \Delta(q)=q\prod_{n\geq1}(1-q^n)^{24}.
\]
The limiting Weierstrass equation has a unit discriminant in $\cU^\epsilon$ and its restriction to $H(\omega)$ is precisely the Tate equation with parameter $q_\omega$. Since $H(\omega)$ is complete and algebraically closed,  \cite[Ch V, Thm 3.1]{SilvermanAdvanced} gives an elliptic curve
\[
E_\omega\simeq H(\omega)^\times/q_\omega^{\mathbb Z}
\]
with split multiplicative reduction. This is the required ultralimit of the $E_j$.
\end{proof}

Let $A_j=E_j\times F_j$ and $\epsilon_j:= \lambda_{\partial X}(A_j)^{-1}$. Equip the products with the  full $\ell_p$ level structures, and assume that they form a centered sequence with fundamental ultralimit $A_\omega/H(\omega)$. If 
\begin{equation}\label{eq wLimit bounded}
0<\lim_\omega\eps_j\nu_j(q_j)<\infty, \qquad 0<\lim_\omega\eps_j\nu_j(q'_j)<\infty,
\end{equation}  the previous lemma applies to either sequence to get elliptic curves $E_\omega$ and $F_\omega$ over $H(\omega)$ with Tate parameters $q_\omega, q_\omega ' \in H(\omega)$ respectively. 

\begin{lemma}\label{lem aw is ew x fw}
Assume equation (\ref{eq wLimit bounded}) holds.
Then $A_\omega\simeq E_{\omega}\times F_{\omega}$. Moreover, 
\[
\Sk(A_\omega)= \Sk(E_\omega)\times \Sk (F_\omega)
\]
\end{lemma}
\begin{proof}
On an $\omega$-large set both rescaled Tate lengths are bounded above and bounded away from zero. Lemma~\ref{lem:tatecurveultralimit} constructs elliptic schemes $\mathscr E$ and $\mathscr F$ over $U^\eps$. Define $\mathscr B=\mathscr E\times\mathscr F$, with the product polarization. We have isomorphisms $E_j\simeq\mathscr E_{C_j}$ and $F_j\simeq\mathscr F_{C_j}$, hence $\mathscr B_{C_j}\cong A_j$. The four elements of the $\ell_p$ torsion forming a basis for the level structure of $A_j$ therefore can be considered over $\sB_{C_j}$. The order of these four elements and the pairing on them are finite data, therefore, there is an $\omega$-large set on which the level-structure is fixed. Since $\sB$ is projective over $U^\epsilon$, we can lift these four elements of $\sB_{C_j}$ to sections of $\sB$ over $U^\epsilon$. These lifted sections are killed by $\ell_p$ and their pairings are the fixed pairing mentioned above and thus define a full symplectic level structure. We now have a $U^\eps$-point of $X$ representing $\mathscr B$ with this level structure. Its $C_j$-fibers are the original moduli points of $A_j$. By the uniqueness in Lemma \ref{lem:boundary-degeneration}, we get that $A_\omega$ is isomorphic to $\sB_\omega$ which is $E_\omega \times F_\omega$.

Finally, the analytic uniformization of this product is $(\Gm^2)^{\an}/\langle(q_\omega,1),(1,q'_\omega)\rangle$. Its tropicalization is the quotient of $\R^2$ by the lattice
\[
\nu_\omega(q_\omega)\Z\oplus\nu_\omega(q'_\omega)\Z.
\]
The canonical skeleton is the image of the Gauss section of this tropicalization \cite{gubler2023monge}. The Gauss seminorm at $(s,t)$ is the product-torus Gauss seminorm with those two coordinate valuations. Thus this skeleton is the product of the two circle skeleta $\Sk(E_\omega)\times \Sk(F_\omega)$, hence the retraction is coordinatewise. This proves the remaining assertions.
\end{proof}

\begin{lemma}\label{lem w of limit is limit of w}
With the assumption as in lemma \ref{lem aw is ew x fw}, there are representatives $(z_{j},z_{j}')$ and $(z_\omega, z'_\omega)$ in the Tate coordinate of $P_j$ and $P_\omega$ such that
\[
\frac{\nu_\omega(z_{\omega})}{\nu_\omega(q_{\omega})} =\lim_\omega\frac{\nu_j(z_{j})}{\nu_j(q_{j})}, \qquad \frac{\nu_\omega(z_{\omega}')}{\nu_\omega(q_{\omega}')} =\lim_\omega\frac{\nu_j(z_{j}')}{\nu_j(q_{j}')}.
\]
\end{lemma}

\begin{proof}
It is enough to treat the first factor and the second factor follows similarly. Multiplication of a representative by a power of $q_j$ allows us to have
\[
-\tfrac 1 2 \nu_j(q_j)\leq\nu_j(z_j)\leq\tfrac12\nu_j(q_j).
\]
Or equivalently $|q_j|_j^{1/2}\leq|z_j|_j\leq|q_j|_j^{-1/2}$. On an $\omega$-large set $J$, we can choose $0<a<b$ with
\[
 a\leq\eps_j\nu_j(q_j)\leq b.
\]
Then $(z_j)$ and $(z_j^{-1})$ both belong to $U^\eps$, with rescaled norms bounded by $e^{b/2}$. Their images are inverse elements, so $z_\omega$ is a nonzero element of $H(\omega)$. Put
\begin{align*}
 R_X(q,z)&=\sum_{n\geq1}\left(
 \frac{q^nz}{(1-q^nz)^2}
 +\frac{q^nz^{-1}}{(1-q^nz^{-1})^2}
 -\frac{2q^n}{(1-q^n)^2}\right),\\
 R_Y(q,z)&=\sum_{n\geq1}\left(
 \frac{(q^nz)^2}{(1-q^nz)^3}
 -\frac{q^nz^{-1}}{(1-q^nz^{-1})^3}
 +\frac{q^n}{(1-q^n)^2}\right).
\end{align*}
Tate's formulas \cite[Chap V, Thm 3.1]{SilvermanAdvanced} give the affine coordinates
\[
 X(q,z)=\frac{z}{(1-z)^2}+R_X(q,z),
 \qquad
 Y(q,z)=\frac{z^2}{(1-z)^3}+R_Y(q,z).
\]
Set $r=e^{-a}<1$. For $n\geq1$ the chosen representatives satisfy
\[
 |q_j^nz_j|_j^{\eps_j},\ |q_j^nz_j^{-1}|_j^{\eps_j}
 \leq r^{n-1/2}<1.
\]
All denominators in the two series consequently have norm $1$. Each summand has rescaled norm at most $r^{n-1/2}$, and each tail with $n>N$ has rescaled norm at most $r^{N+1/2}$. The uniform tail estimate therefore proves
\[
(R_X(q_j,z_j))_\omega=R_X(q_\omega,z_\omega), \qquad (R_Y(q_j,z_j))_\omega=R_Y(q_\omega,z_\omega).
\]
To include the identity, we  use the homogeneous tuple
\[
P(z,q):= \left[
 z(1-z)+(1-z)^3R_X(q,z):
 z^2+(1-z)^3R_Y(q,z):
 (1-z)^3
 \right].
\]
The formula for any coordinate of $P(z,q)$ commutes with the ultralimit. Thus $z_\omega = \lim_\omega z_j$ and 
\[
\nu(z_\omega )= \lim_\omega \epsilon_j \nu_j(z_j)
\]
By Lemma \ref{lem:tatecurveultralimit}, we have $\nu(q_\omega )= \lim_\omega \epsilon_j \nu_j(q_j)>0$. 
The value $\epsilon_j \nu_j(z_j)$ is bounded above, and the denominators $\epsilon_j\nu_j(q_j)$ are bounded away from zero on an $\omega$-large set. Taking the quotient proves the Lemma identity for these representatives.  
\end{proof}

Let $Y:= X_{1,\ell_p}\times X_{1,\ell_p}$ where $X_{1,\ell_p}$ is the moduli space of elliptic curves with a level-$\ell_p$ structure and let $X:=X_{2, \ell_p}$. Let $\overline X_{1,\ell_p}$ be the compactified modular curve. We use the compactification $\ol Y:= \overline X_{1,\ell_p}\times \overline X_{1,\ell_p}$. The $j$-map is finite over $\P^1$, and its boundary is the inverse image of $\infty$. See \cite[Sec 8.6.3, Thm 8.6.8]{katzmazur1985}. On $\overline Y=\overline X_{1,\ell_p}\times\overline X_{1,\ell_p}$ put
\[
 D=j_1^*(\infty)+j_2^*(\infty).
\]
This effective Cartier divisor has support $\overline Y\setminus Y$. We use $D$ to define $\lambda_{\partial Y}$. 

\begin{lemma}\label{lem f.m. map of Y to X}
 There are finitely many morphisms $f_1,\cdots , f_r :Y\to X$ over $\kappa_p$ such that every geometric point of $X$ whose principally polarized abelian surface is isomorphic to a product with its product principal polarization belongs to one of their images.
\end{lemma}
\begin{proof}
The product of the two universal elliptic curves, their principal polarizations, and their level structures defines a morphism $f_0:Y\to X$. For each
\[
 \gamma\in\Sp_4(\Z/\ell_p\Z),
\]
precompose the product level structure with $\gamma$ to obtain $f_\gamma$. Let $x$ be a point of $X$ that corresponds to a product of elliptic curves, where the principal polarization comes from the product of the principal polarizations of the elliptic curves. Over an algebraically closed field, choose symplectic level bases on the two elliptic factors of a polarized product. Their concatenation is a symplectic basis of the product torsion. Any other prescribed symplectic basis differs by an element of $\Sp_4(\Z/\ell_p\Z)$, proving the lemma. 
\end{proof}

\begin{lemma}\label{lem agamma for any f}
Let $f$ be one of the maps from the previous Lemma. Denote $T$ to be the reduced closure of the graph of $f$ inside $\ol Y \times \ol X$ with natural projections $\pi_X : T\to \ol X$ and $\pi_Y:T\to \ol Y$. Let $\cJ:=\pi_X^*\cI_{\partial X}$ and $\cK:=\pi_Y^*\cI_{\partial Y}$. There are integers $a, b>0 $ that depend only on $f$ such that 
\[
\cJ^a \subset \cK, \qquad \cK^b \subset \cJ
 \]
\end{lemma}
\begin{proof}
Suppose $t\in T$ maps to $\overline Y\setminus Y$ but $\pi_X(t)\in X$. We will show that such a point doesn't exist, so it's enough to work over an algebraically closed field.  The closed set $T\setminus \pi_Y^{-1}(Y)$ has at most dimension $1$. Let $U =\Spec A \subset T$ be an open affine neighborhood around $t$ and let $\fm$ be the maximal ideal defining $t$. Write the one dimensional components of the intersection $U \cap \big (T\setminus \pi_Y^{-1}(Y) \big )$ as $B_1, \cdots ,B_r$, and let $\fp_1, \cdots \fp_r$ be their corresponding prime ideals. Since $\fm$ is a maximal ideal of height 2, we have that 
\[
\fm \not\subset \cup_{i=1}^r \fp_i
\]
Let $g\in \fm \setminus \cup_{i=1}^r \fp_i$, then an irreducible component of $V(g)$ is a curve $B'$ passing through $t$ but it is not one of the components $B_1,\cdots,B_r$. Let $s$ be a point of the normalization of the curve $B'$ that maps to $t$. Let $R$ be the local ring at $s$. Then $R$ is a discrete valuation ring with a map $h:\Spec R \to T$. The generic point of $\Spec R$ maps to $Y$ and $X$ while the special point maps to $\ol Y \setminus Y$ and $X$ under $\pi_Y \circ h$ and $\pi_X \circ h$. Trivial action of the inertia group restricts to the trivial action on any factors. Hence the N\'eron-Ogg-Shafarevich Theorem \cite[Thm 1]{serre1971good} gives a contradiction and no such $t$ exists. Moreover $X$ and $Y$ are  separated, $\pi_Y^{-1}(Y)$ is the graph of $f$ and maps into $X$. Hence $\pi_Y^{-1}(\ol Y \setminus Y) = \pi_X^{-1}(\ol X \setminus X) $ and $\cJ$ and $\cK$ have the same support. Since $T$ is Noetherian,  there exist $a,b>0$ depending only on $f$ coming from the structure sheaf of $\cJ$ and $\cK$ such that
\[
\cJ^a\subset \cK , \qquad \cK^b \subset \cJ
\]
\end{proof}

\begin{proposition}\label{prop c and C}
There exist $c, c'>0$ such that for any Tate elliptic curves $E$ and $F$ with Tate lengths $L_E=\nu(q_E)$ and $L_F=\nu(q_F)$ over a complete algebraically closed valued field $C$ containing $\kappa_p$ trivially, we have 
\[
c(L_E+L_F)\le \lambda_{\partial_X}(E\times F) \le c'(L_E+L_F)
\]
\end{proposition}

\begin{proof}
By \cite[Chap V, Thm 3.1]{SilvermanAdvanced} the $j$-invariant of a Tate elliptic curve with parameter $q$ is given by an integral expansion $j(q)=q^{-1}+\sum_{n\ge 0} c_nq^n$. Note that $|q|<1$ and $|j(q)|=|q|^{-1}$. Consequently the boundary function of $D$ at a pair of Tate curves is
\begin{equation}\label{eq:D-length}
\lambda_D(E\times F)=\log^+|j(E)|+\log^+|j(F)| =\nu(q_E)+\nu(q_F).
\end{equation}
Let $f_\gamma$ be one of the maps from Lemma \ref{lem f.m. map of Y to X} with the reduced graph $T_\gamma$, ideal sheaves $\cJ_\gamma$ and $\cK_\gamma$ on $T_\gamma$, and constants $a_\gamma$ and $b_\gamma$ from Lemma \ref{lem agamma for any f}. Let $x\in X$ correspond to the product of two Tate elliptic curves $E\times F$. By Lemma \ref{lem f.m. map of Y to X}, there exists $\gamma$ such that for some $y\in Y$, we have $f_\gamma(y)=x$.  Let $R$ be the valuation ring of C. Since $T_\gamma \to \ol Y$ is proper, the morphism $\Spec R \to \ol Y$ extends to a morphism $\Spec R\to T_\gamma$. By the containment of the ideals $\cJ_\gamma$ and $\cK_\gamma$ from Lemma \ref{lem agamma for any f}, we have that 
\[
\frac{1}{a_\gamma}\lambda_D(E\times F) \leq\lambda_{\partial X}(E\times F) \leq b_\gamma\lambda_D(E\times F).
\]
Taking $c=\min_\gamma a_\gamma^{-1}$ and $c'=\max_\gamma b_\gamma$ finishes the proof. 
\end{proof}

\begin{proposition}\label{prop:fm-admissible}
Fix the characteristic exponent $p$ and an integer $M\ge1$. For local products $E\times F$ having multiplicative reduction, with integer Tate lengths  $\rho u_E$ and $\rho u_F$ where $ u_E, u_F\in\{1,\dots,M\}$ and $\rho>0$, the signed diagonals $\Delta_{u_E,u_F}^{\pm}$ form a level-stably admissible system.
\end{proposition}
\begin{proof}
Consider a centered sequence $A_j:=E_j\times F_j$ where $E_j$ and $F_j$ are Tate elliptic curves with Tate lengths $\rho u_{E_j}$ and $\rho u_{F_j}$ and let $\epsilon_j:=\lambda_{\partial X}(A_j)^{-1}$. By assumption and Proposition \ref{prop c and C}, we have 
\[
\frac{1}{c'(M+1)}\le \epsilon_j \rho u_{E_j}\, \le \frac 1 c, \qquad  \frac{1}{c'(M+1)}\le \epsilon_j \rho u_{F_j}\, \le \frac 1 c
\]
This is exactly equation (\ref{eq wLimit bounded}), thus Lemma \ref{lem aw is ew x fw} implies $\Sk(A_\omega)= \Sk(E_\omega)\times \Sk (F_\omega)$. Let $E_\omega$ and $F_\omega$ have Tate parameters $q_\omega$ and $q_\omega'$. Note that a point $P_j$ given in coordinates $(z_j, z_j')$ reduces to $\Delta_j$ if and only if $\nu_j(z_j) /  \rho u_{E_j} = \pm (\nu_j(z_j') / \rho u_{F_j})$. Let $P_j$ be a sequence of such points and let $P_\omega$ be its ultralimit given by $(z_\omega, z_\omega')$. By Lemma \ref{lem w of limit is limit of w}, we have 
\[
\frac{\nu_\omega(z_{\omega})}{\nu_\omega(q_{\omega})} =\lim_\omega\frac{\nu_j(z_{j})}{\nu_j(q_{j})}
\]
Thus  $\nu_\omega(z_{\omega}) / \nu_\omega(q_{\omega}) = \pm \nu_\omega(z_{\omega}') /\nu_\omega(q_{\omega}')$ and the point $P_\omega$ reduces to the weighted diagonal of $\Sk(E_\omega)\times \Sk (F_\omega)$. This shows that the prescribed family is level-stably admissible.
\end{proof}

\section{Proof of Theorem \ref{thm:main}}\label{sec divisors}\label{sec:proof}

\subsection{Divisors}

This subsection is independent of the function-field geometry. Throughout this subsection, let $K$ be any field and fix an algebraic closure $\ol K$. We will discuss what happens when a divisor of $E\times F$ contains a finite subgroup defined over $K$. 
\begin{definition}
Define
\[
K^\natural=
\begin{cases}
K,&\charac(K)=0,\\
K^{\perf},&\charac(K)>0,
\end{cases}
\]
where $K^{\perf}$ is the perfect closure of $K$ in $\overline K$.
\end{definition}

\begin{lemma}
\label{lem:elliptic-subgroup-descent}
Let $E$ and $F$ be elliptic curves over $K$ and let $A:=E\times F$.  Let $H\subset A_{\overline K}$ be an elliptic subgroup scheme stable under $\Gal(\ol K/ K)$. Then $H$ descends to an elliptic subgroup $H_0\subset A_{K^\natural}$. Moreover, if both projections $H_0\to E_{K^\natural}$ and $H_0\to F_{K^\natural}$ are nonzero, then $E$ and $F$ are isogenous over $K$.
\end{lemma}

\begin{proof}
The fixed field of $\Gal(\Kbar/K)$ in $\Kbar$ is $K$ in characteristic zero and $K^{\perf}$ in positive characteristic.  Standard Galois descent for the closed subgroup scheme $H\subset A_{\Kbar}$ therefore gives $H_0$ over $K^\natural$.  Its smoothness, connectedness, and group structure descend
faithfully flatly.

If both projections are nonzero,  $E, F$, and $H_0$ are isogenous over $K^\natural$.  The extension $K^\natural/K$ is primary, so Chow's Theorem \cite[Lem~1.2.1.2]{chaiconoort} descends this isogeny to $K$.
\end{proof}

Fix the following data for the rest of this subsection.  Let $E, F$ be nonisogenous elliptic curves over $K$.   Fix an ample line bundle $L$ on $A:=E\times_K F$ and let $d:=L^2$. For any $a\in A(\ol K)$ let $\tau_a$ denote the translation by $a$.

\begin{lemma}\label{lem no common component}
Let $\Gamma$ be a finite $\Gal(\ol K/K)$-invariant subgroup of $A_{\operatorname{tor}}(\ol K)$ such that  $|\Gamma|=n$ and $(n, p)=1$.  Let $D\in |L^{\otimes m}|$ be an effective divisor containing $\Gamma$.  If for some $\gamma\in \Gamma$ two effective divisors $D$ and $\tau_\gamma D$ do not share any irreducible component, then 
 \[
  n\le m^2 d
 \]
\end{lemma}

\begin{proof}
Write
\[
D=\sum_{i=1}^r m_iC_i,
\]
where $C_i$ are the distinct integral components of the support and $m_i\geq 1$. Translation on an abelian variety acts trivially on the N\'eron--Severi group \cite[Sec 6]{mumford1974abelian}. Thus $D\cdot\tau_\gamma D=D^2$. Since the two divisors have no common component, their intersection is proper. Each of the $n$ distinct points of $\Gamma$ contributes local intersection multiplicity at least one, so
\[
n\leq D\cdot\tau_\gamma D=D^2= m^2 L^2=m^2d.
\]
\end{proof}

\begin{lemma}\label{lem finite stabilizer}
Let $C\subset A_{\Kbar}$ be an irreducible curve, let $\mathcal S=\Stab(C)$ be its scheme-theoretic stabilizer, and write
\[
 S=\mathcal S(\Kbar)=\{a\in A(\Kbar):\tau_aC=C\}.
\]
Let $G\subset S\cap A_{\operatorname{{tor}}}(\Kbar)$ be a finite subgroup of order prime to $p$. If $S$ is finite, then
\[
 |G|\le\frac{C^2}{2}.
\]
\end{lemma}
\begin{proof}
Let $\widetilde C$ be the normalization of $C$. There are no nonconstant maps from $\P^1$ into any abelian varieties \cite[Cor 3.8]{milne1986abelian} and the stabilizer of a genus 1 curve has dimension 1 because a curve of geometric genus one in an abelian variety is a translate of an elliptic subgroup \cite[Cor~2.2]{milne1986abelian}.  Hence the genus of  $\Tilde{C}$ is at least two.  The prime-to-$p$ group $G$ acts freely on $\widetilde C$, since a nonzero translation has no fixed point on $A$.  The quotient is therefore an \'{e}tale cover of degree $|G|$. Let $g'$ be the genus of the quotient curve.  Hurwitz formula gives
\[
 2g(\widetilde C)-2=|G|(2g'-2).
\]
Since the left side is positive, $g'\ge2$, so $|G|\le g(\tilde C)-1$. Adjunction on the abelian surface, whose canonical class is trivial, gives $2p_a(C)-2=C^2$. As $g(\tilde C)\le p_a(C)$, it follows that $|G|\le C^2/2$.   These are the usual Hurwitz and adjunction formulas \cite[Chap~IV, Cor~2.4, Chap~V,
Prop~1.5]{hartshorne2013algebraic}.  
\end{proof}

\begin{theorem}\label{thm zero-lemma}
Let $(N,p)=1$ and let $\psi:E[N]\xrightarrow{\sim}F[N]$ be $G_K$-equivariant.  Set
\[
\Gamma=\{(x,\psi(x)):x\in E[N]\}\subset A[N].
\]
If an effective divisor $D\in|L^{\otimes m}|$, defined over $K$, contains $\Gamma$, then
\begin{equation}\label{eq:zero-bound}
 N\leq Z(m,d):= \max\left\{ m\sqrt d,\ md,\ \frac{m^2d^{3/2}}{\sqrt2},\ \sqrt2\,m^3d^{5/2} \right\}.
\end{equation}
\end{theorem}

\begin{proof}
Write \( D=\sum_{i=1}^r m_iC_i\) with distinct integral curves $C_i$ over $\ol K$ and $m_i\geq1$.  Since $L$ is ample, $L\cdot C_i$ is a positive integer, and
\begin{equation}\label{eq:r-bound}
 r\leq\sum_i m_i(L\cdot C_i)=L\cdot D=md.
\end{equation}
Suppose first that for some $\gamma\in\Gamma$, the divisors $D$ and $\tau_\gamma D$ have no common component.  Then Lemma \ref{lem no common component} proves that
\[
 N^2\leq m^2d,
\]
which gives the first term in \eqref{eq:zero-bound}. We may now assume that $D$ and $\tau_\gamma D$ share a component for every $\gamma\in\Gamma$.  For $1\leq i,j\leq r$, put
\[
 T_{ij}=\{a\in A(\Kbar):\tau_aC_i=C_j\}, \qquad S_i=T_{ii}=\Stab(C_i).
\]
Then $\Gamma\subset\bigcup_{i,j}T_{ij}$.  If $T_{ij}$ is nonempty, it is a coset of $S_i$.  Hence some $i$ satisfies
\begin{equation}\label{eq:s-bound}
 s:=|\Gamma\cap S_i|\geq\frac{N^2}{r^2}.
\end{equation}
The identity component $S_i^0$ has dimension zero or one. Assume first that $S_i$ is finite.  By Lemma \ref{lem finite stabilizer} and the Hodge index theorem 
\[
s\le \frac{C_i^2}{2}\leq\frac{(L\cdot C_i)^2}{2d}\leq \frac{m^2d}{2}.
\]
Together with \eqref{eq:s-bound} and \eqref{eq:r-bound}, this yields
\[
 N^2\leq r^2s\leq\frac{m^4d^3}{2},
\]
which is the third term in \eqref{eq:zero-bound}. It remains to consider $\dim S_i^0=1$.  Put $H=(S_i^0)_{\red}$, the reduced scheme of the identity component of the $S_i$.  The orbit of any point of $C_i$ under $H$ is a closed integral curve contained in $C_i$, so $C_i$ is a translate of $H$.  It follows that $S_i=H$.  If one projection of $H$ to $E$ or $F$ is zero, then $H$ is a coordinate factor.  The graph $\Gamma$ meets a coordinate factor only at the origin, so $s=1$.  Equations \eqref{eq:s-bound} and \eqref{eq:r-bound} give $N\leq r\leq md$, the second term in \eqref{eq:zero-bound}.

Suppose both projections are nonzero.  Since $D$ is defined over $K$, elements of the Galois group permute its components.  Thus every conjugate $H^\sigma$ is the stabilizer of another component and $|\Gamma\cap H^\sigma|=s$. If no conjugate differs from $H$, then $H$ is Galois-stable. Lemma \ref{lem:elliptic-subgroup-descent} and the nonzero projections imply that $E$ and $F$ are isogenous over $K$, contrary to the hypothesis. Therefore, for some $\sigma$ we have  $H^\sigma\neq H$ and the subgroup formula inside the group $\Gamma$ of order $N^2$ gives
\[
 |\Gamma\cap H\cap H^\sigma| \ge \frac{|\Gamma\cap H||\Gamma\cap H^\sigma|}{|\Gamma|}=\frac{s^2}{N^2}\geq\frac{N^2}{r^4}.
\]
The left side is at most the intersection number $H\cdot H^\sigma$. An elliptic subgroup has self-intersection zero, so $H^2=(H^\sigma)^2=0$. The Hodge index Theorem applied to $H+H^\sigma$ gives
\[
 2d(H\cdot H^\sigma) \leq\bigl( L\cdot(H+H^\sigma) \bigr)^2 \leq(2md)^2,
\]
so $H\cdot H^\sigma\leq2m^2d$.  Consequently,
\[
 N^2\leq2m^2dr^4\leq2m^6d^5,
\]
which is the last term in \eqref{eq:zero-bound}.
\end{proof}

\subsection{Proof of the main theorem}

\begin{proof}[ Proof of Theorem \ref{thm:main}]
    
  Fix $b$ and
$p$, and set
\[
 M=48(b+1).
\]
By Proposition \ref{prop:perfect-constant}, we may work over a perfect constant field.  In
positive characteristic, Proposition \ref{prop:frob-minimal} replaces $E$ and $F$ by elliptic curves with generically separable $j$-invariants without changing prime-to-$p$ torsion.  If these new elliptic curves were isogenous, then the original curves would be isogenous as well.  We may therefore assume, when proving a bound for nonisogenous pairs, that both $j$-maps are generically separable.

Consider first a prime-power level $Q=\ell^\alpha$.  Levels $Q\leq M^2$ already satisfy a uniform bound. Suppose $Q>M^2$.  Proposition~\ref{prop:smallplaces} and Proposition~\ref{prop:diagonal-retraction} give a lowered level $Q_0$, integers $u_E,u_F\leq M$, and a set $\Sigma$ of jointly multiplicative places satisfying
\begin{equation}\label{eq:final-density}
 \deg_k\Sigma\geq \frac{h_\omega(E)+h_\omega(F)}{16(b+1)M^3}.
\end{equation}
After the scalar normalization $\psi := \phi \circ [a]$ in Proposition \ref{prop:a0 exists}, the graph $\Gamma_\psi$ is universally restricted at every $v\in\Sigma$ by the signed diagonal $\Delta^\pm_{u_E,u_F}$. Let $\lambda_E$ and $\lambda_F$ be principal polarizations on the elliptic curves $E,F$. Put $A=E\times F$ and give it the product principal polarization $\lambda:= \lambda_E \times \lambda_F$. Define 
\[
 M_\lambda:=(\operatorname{id}_A,\lambda)^*\mathcal P_A, \qquad L_A:=M_\lambda^{\otimes3}.
\]
Let $\Theta:=E\times\{0_F\}+\{0_E\}\times F$, then the product principal polarization is represented by $\mathcal O_A(\Theta)$, and $M_\lambda$ represents $2\lambda$ as in Section~1. Thus $L_A\equiv6\Theta$. The two coordinate factors have self-intersection zero and meet transversely at the origin, so $\Theta^2=2$ and $d:=L_A^2=72$. We have
\[
h_\omega(A)=h_\omega(E)+h_\omega(F)>0.
\]
Take
\[
 M_0=16(b+1)M^3.
\]
Let $\mathcal Z_M$ be the system of signed diagonals $\Delta_{u_E,u_F}^{\pm}$ with  $1\le u_E,u_F\le M$. By  Proposition \ref{prop:fm-admissible}, $\mathcal Z_M$ is level-stably admissible. Apply Theorem \ref{thm:heightgap} with abelian dimension $2$, characteristic exponent $p$, the ratio $M_0$, and the system $\mathcal Z_M$. This gives a constant $q\ge2$  depending only on $b$ and $p$ such that for the universally $\Sigma$-restricted set of torsion points $\Gamma_\psi$ we have that either $|\Gamma_\psi|< h^0(A,L^{\otimes q})$ or a nonzero section of $L^{\otimes q}_{\ol K}$ vanishes on $\Gamma_\psi$. Riemann--Roch on the abelian surface gives 
\[
 h^0(A,\cL_A^{\otimes q})=\frac{q^2d}{2} = 36q^2. 
\]

Since $\psi$ is an isomorphism on $Q_0$-torsion, the set $\Gamma_\psi$ consists of $Q_0^2$ many points. If $Q_0^2<h^0(A,L^{\otimes q})$, then $Q_0<6q$. Now assume that a nonzero section of $L^{\otimes q}_{\ol K}$ vanishes on $\Gamma_\psi$. Every point of $\Gamma_\psi$ is prime-to-$p$ torsion, hence separable. Theorem \ref{thm:descent} gives an effective divisor $D\in|L_A^{\otimes q}|$, defined over $K$, that contains $\Gamma_\psi$. By Theorem \ref{thm zero-lemma},
\[
 Q_0\le Z(q,72).
\]
Together with \eqref{eq:level-ratio}, every nonisogenous pair at a prime-power level satisfies
\[
 Q\le C_0(b,p):=\max\{M^2,Mq\sqrt{72/2},MZ(q,72)\}.
\]

Finally, let $N$ be arbitrary and prime to $p$.  For every prime power $\ell^\alpha\Vert N$, an isomorphism $E[N]\simeq F[N]$ restricts to $E[\ell^\alpha]\simeq F[\ell^\alpha]$.  Apply Lemma \ref{lem:primary-decomp} with $C_0=C_0(b,p)$ and define
\[
 C(b,p)=C_1(C_0(b,p),p).
\]
Thus a nonisogenous pair can only occur when $N\leq C(b,p)$. This finishes the proof of Theorem \ref{thm:main}.
\end{proof}

\section*{Acknowledgment}
I would like to thank my advisor, Jacob Tsimerman, for mentioning this question and the work of Looper and Yap, for his guidance and support, and for many valuable discussions. I am also grateful to Jit Wu Yap for suggesting the extension from characteristic zero to arbitrary characteristic, and for emphasizing the importance of verifiable admissibility criteria. I'd also like to thank Hamid Abban for his support during the preparation of this paper. This research has been partially funded by the Royal Society International Award $ICA\backslash R 1\backslash231019$.

 \bibliographystyle{alpha}
 \bibliography{Main}

\begin{thebibliography}{BGM22}

\bibitem[AGV72]{artin1973theorie}
Michael Artin, Alexandre Grothendieck, and Jean~Louis Verdier.
\newblock {\em Th{\'e}orie des topos et cohomologie etale des sch{\'e}mas: Tome 2}.
\newblock Springer-Verlag, 1972.

\bibitem[Ber90]{berkovich1990spectral}
Vladimir~G Berkovich.
\newblock {\em Spectral theory and analytic geometry over non-Archimedean fields}.
\newblock Number~33. American Mathematical Soc., 1990.

\bibitem[BG07]{bombieri2007heights}
Enrico Bombieri and Walter Gubler.
\newblock {\em Heights in Diophantine geometry}, volume~4.
\newblock Cambridge university press, 2007.

\bibitem[BGM22]{boucksom2022differentiability}
S{\'e}bastien Boucksom, Walter Gubler, and Florent Martin.
\newblock Differentiability of relative volumes over an arbitrary non-archimedean field.
\newblock {\em International Mathematics Research Notices}, 2022(8):6214--6242, 2022.

\bibitem[Bil99]{billingsley1999convergence}
Patrick Billingsley.
\newblock {\em Convergence of probability measures}.
\newblock wiley New York, 2nd edition, 1999.

\bibitem[BL84]{bosch1984stable}
Siegfried Bosch and Werner L{\"u}tkebohmert.
\newblock Stable reduction and uniformization of abelian varieties ii.
\newblock {\em Inventiones mathematicae}, 78(2):257--297, 1984.

\bibitem[BLR90]{BLR}
Siegfried Bosch, Werner L{\"u}tkebohmert, and Michel Raynaud.
\newblock {\em N{\'e}ron models}.
\newblock Springer Science \& Business Media, 1990.

\bibitem[BT16]{bakker2016p}
Benjamin Bakker and Jacob Tsimerman.
\newblock p-torsion monodromy representations of elliptic curves over geometric function fields.
\newblock {\em Annals of Mathematics}, pages 709--744, 2016.

\bibitem[CCO13]{chaiconoort}
Ching-Li Chai, Brian Conrad, and Frans Oort.
\newblock {\em Complex multiplication and lifting problems}, volume 195.
\newblock American Mathematical Soc., 2013.

\bibitem[CL03]{chambert2003mesures}
Antoine Chambert-Loir.
\newblock Mesures et équidistribution sur les espaces de berkovich.
\newblock {\em arXiv preprint math/0304023}, 2003.

\bibitem[FC13]{faltings2013degeneration}
Gerd Faltings and Ching-Li Chai.
\newblock {\em Degeneration of abelian varieties}.
\newblock Springer Science \& Business Media, 2013.

\bibitem[FG25]{favre2025non}
Charles Favre and Chen Gong.
\newblock Non-archimedean techniques and dynamical degenerations.
\newblock {\em Peking Mathematical Journal}, pages 1--63, 2025.

\bibitem[GP22]{GP22}
Richard Griffon and Fabien Pazuki.
\newblock Isogenies of elliptic curves over function fields.
\newblock {\em International Mathematics Research Notices}, (19):14697--14740, 2022.

\bibitem[GS23]{gubler2023monge}
Walter Gubler and Stefan Stadl{\"o}der.
\newblock Monge--amp{\`e}re measures for toric metrics on abelian varieties.
\newblock {\em Publications math\'{e}matiques de Besan\c{c}on. Alg\`{e}bre et th\'{e}orie des nombres}, pages 49--84, 2023.

\bibitem[Gub10]{gubler2010non}
Walter Gubler.
\newblock Non-archimedean canonical measures on abelian varieties.
\newblock {\em Compositio Mathematica}, 146(3):683--730, 2010.

\bibitem[Har92]{JoeHarris1992Alg.Geom}
Joe Harris.
\newblock {\em Algebraic Geometry}.
\newblock Springer New York, 1992.

\bibitem[Har13]{hartshorne2013algebraic}
Robin Hartshorne.
\newblock {\em Algebraic geometry}.
\newblock Springer Science \& Business Media, 2013.

\bibitem[KM85]{katzmazur1985}
Nicholas~M. Katz and Barry Mazur.
\newblock {\em Arithmetic Moduli of Elliptic Curves}.
\newblock Number 108 in Annals of Mathematics Studies. Princeton University Press, Princeton, NJ, 1985.

\bibitem[Loo24]{looper2024arakelov}
Nicole~R Looper.
\newblock Arakelov-green's functions for dynamical systems on projective varieties.
\newblock {\em arXiv preprint arXiv:2404.06981}, 2024.

\bibitem[LR19]{lozano2019number}
{\'A}lvaro Lozano-Robledo.
\newblock {\em Number theory and geometry: An introduction to arithmetic geometry}, volume~35.
\newblock American Mathematical Soc., 2019.

\bibitem[LY26]{looper2026uniform}
Nicole Looper and Jit~Wu Yap.
\newblock Uniform boundedness of small points on abelian varieties over function fields.
\newblock {\em arXiv preprint arXiv:2603.23396}, 2026.

\bibitem[MG78]{MazurGoldfeld1978}
B.~Mazur and D.~Goldfeld.
\newblock Rational isogenies of prime degree.
\newblock {\em Inventiones Mathematicae}, 44(2):129–162, 1978.

\bibitem[Mil86]{milne1986abelian}
James~S Milne.
\newblock Abelian varieties.
\newblock {\em Arithmetic geometry}, pages 103--150, 1986.

\bibitem[Mum70]{mumford1974abelian}
David Mumford.
\newblock {\em Abelian varieties}, volume~5.
\newblock Oxford university press Oxford, 1970.

\bibitem[Poo13]{poonen2013extending}
Bjorn Poonen.
\newblock Extending self-maps to projective space over finite fields.
\newblock {\em Documenta Mathematica}, 18:1039--1044, 2013.

\bibitem[PS00]{pesenti2000inegalite}
J{\'e}r{\^o}me Pesenti and Lucien Szpiro.
\newblock In{\'e}galit{\'e} du discriminant pour les pinceaux elliptiques {\`a} r{\'e}ductions quelconques.
\newblock {\em Compositio Mathematica}, 120(1):83--117, 2000.

\bibitem[Sek77]{sekiguchi1977projective}
Tsutomu Sekiguchi.
\newblock On projective normality of abelian varieties ii.
\newblock {\em Journal of the Mathematical Society of Japan}, 29(4):709--727, 1977.

\bibitem[Sil94]{SilvermanAdvanced}
Joseph~H. Silverman.
\newblock {\em Advanced Topics in the Arithmetic of Elliptic Curves}, volume 151 of {\em Graduate Texts in Mathematics}.
\newblock Springer, New York, 1994.

\bibitem[Sil09]{silverman2009arithmetic}
Joseph~H Silverman.
\newblock {\em The arithmetic of elliptic curves}, volume 106.
\newblock Springer, 2009.

\bibitem[ST68]{serre1971good}
Jean-Pierre Serre and John Tate.
\newblock Good reduction of abelian varieties.
\newblock {\em Annals of Mathematics}, 88(3):492--517, 1968.

\bibitem[Sti09]{stichtenoth2009algebraic}
Henning Stichtenoth.
\newblock {\em Algebraic function fields and codes}.
\newblock Springer, 2009.

\bibitem[Ulm11]{ulmer2011}
Douglas Ulmer.
\newblock Park {C}ity lectures on elliptic curves over function fields, 2011.

\end{thebibliography}
 
\end{document}